\documentclass[12pt]{article}

\usepackage{amsmath,amsthm,amssymb,amsfonts}
\usepackage{latexsym}
\def\pmb#1{\setbox0=\hbox{$#1$}%
\kern-.025em\copy0\kern-\wd0
\kern.05em\copy0\kern-\wd0
\kern-.025em\raise.0433em\box0}
\usepackage{amssymb}
\newtheorem{theorem}{Theorem} 
\newtheorem{lemma}[theorem]{Lemma} 
\newtheorem{proposition}[theorem]{Proposition}
\newtheorem{corollary}[theorem]{Corollary}
\newtheorem{remark}{Remark} 
\newcommand{\sm}[1]{\mbox{\small $#1$}}
\newcommand{\la}[1]{\mbox{\large $#1$}}
\newcommand{\La}[1]{\mbox{\Large $#1$}}
\newcommand{\LA}[1]{\mbox{\LARGE $#1$}}

\begin{document}

%\begin{frontmatter}



\title{Dirichlet principal eigenvalue
comparison theorems in  geometry with torsion}


\author{Ana Cristina Ferreira$^{\dag}$ and Isabel Salavessa$^{\ddag}$}
\date{}
\protect\footnotetext{\!\!\!\!\!\!\!\!\!\!\!\!\! {\bf MSC 2010:}
Primary: 58C40, 35P15, 53C07
This work was supported by the 
{\it Funda\c{c}\~ao para a Ci\^encia e Tecnologia} (Portugal) through the 
research project PTDC/MAT/118682/2010, and through 
{\it Centro de Matem\'{a}tica da Universidade do Minho} 
[PEst-C/MAT/UI0013/2011, PEst-OE/MAT/UI0013/2014 to A.C.F] 
and {\it Centro de F\'isica e Engenharia de Materiais Avan\c{c}ados} 
[UID/CTM/04540/ 2015 to I.S.]}
\maketitle ~~~\\[-5mm]
\begin{quotation}\noindent
{\footnotesize ${\dag}$ Centro de Matem\'{a}tica, Universidade do Minho, Campus de Gualtar,
P-4710-057 Braga, Portugal. ~e-mail: anaferreira@math.uminho.pt\\
${\ddag}$  Center of Physics and Engineering of Advanced Materials (CeFEMA),
Instituto Superior T\'{e}cnico, Universidade de Lisboa,
Edif\'{\i}cio Ci\^{e}ncia, Piso 3, Av.\ Rovisco Pais, P-1049-001
Lisboa, Portugal. ~e-mail: isabel.salavessa@ist.utl.pt}\\[2mm]
{\small {\bf Abstract:} 
We describe min-max formulas for the  principal eigenvalue
of a $V$-drift Laplacian defined by a vector field $V$ on a geodesic ball
of a Riemannian manifold $N$.  Then we derive comparison results for the 
principal eigenvalue with the one of a spherically symmetric model space 
endowed with a radial vector field, under pointwise comparison  
of the corresponding  radial sectional and Ricci curvatures,  and of the
radial component of the vector fields. 
These results generalize the  known case $V=0$.}
\end{quotation}

\markright{\sl\hfill  Ferreira -- Salavessa \hfill}

\section{Introduction}

Given a vector field $V$ on a $m$-dimensional  Riemannian manifold $(N,g)$, 
the $V$-drift Laplacian $\Delta_Vu=\Delta u-g(V, \nabla u)$ can be introduced 
in the context of Riemannian geometry with torsion.  If $\nabla^V$ is a 
metric connection with vectorial torsion defined by $V$, $\Delta_Vu$ 
is the trace of the covariant derivative of $du$.  If $V=0$, $\Delta_0u$ 
is the usual Laplacian  for the Levi-Civita connection. The purpose of 
this work is twofold. First, to prove the existence of a principal 
eigenvalue of the operator $-\Delta_V$, $\lambda^*_V$, for any vector field 
$V$ on a regular domain $\bar{M}$ of $N$, under the Dirichlet boundary condition. 
Second, to establish a variational principle for $\lambda_V^*$, and 
 use it to obtain comparison results when $\bar{M}$ is a geodesic
ball.
  
In Lemma \ref{nonself}, we show  that there is a weight function $f$, such that 
$\Delta_V$ is self-adjoint with respect to the $L^2$ space with measure weighted 
by $e^{-f}$ if and only if $V=\nabla f$. In this case, $\Delta_V$ is the 
Bakry-\'{E}mery $f$-Laplacian $\Delta_f$. If $\bar{M}=M\cup \partial M$ is a 
compact domain of $N$ with smooth boundary $\partial M$, the spectrum for the 
eigenvalue problem of $-\Delta_f$ with  Dirichlet boundary  condition is a discrete 
sequence of positive real values converging to infinity. Furthermore, each eigenvalue 
has a variational characterization of Rayleigh type, as shown in \cite{LR}. 
On the other hand, any vector field, $V$, which is not a gradient gives rise to 
an operator $\Delta_V$ for which there is no canonically associated Hilbert space 
on which this operator is self-adjoint. As a consequence, standard arguments used 
to establish a variational principle for an eigenvalue may not be applied. 

We obtain existence of a principal eigenvalue $\lambda_V^*$ using Krein-Rutman theory for compact  operators on $C_0^{1,\alpha}(\bar{M})$. This eigenvalue is a distinguished one, simple, with a positive eigenfunction $\omega_V$, only vanishing on $\partial M$. The eigenvalue $\lambda_V^*$ is positive by a maximum principle argument. In Proposition \ref{MaxPrinc*}, we show that $-\Delta_V$ and its formal adjoint operator $-\Delta^*_V$ have the same set of eigenvalues, which form a discrete set that may only accumulate
 at infinity. Furthermore, they have the same principal eigenvalue. As a consequence,
 the weak maximum principle also holds for $\Delta^*_V$. 

In Theorem \ref{Barta11}, using principal eigenfunctions, we give
a simple proof of Barta's type inequalities (\ref{Barta1})-(\ref{Barta2}) for 
$\lambda^*_V$ on a regular domain $\bar{M}$.
This is a well known inequality for the case $V=0$ (\cite{Cha}, III.2., 
Lemma 1), and it is a useful tool for estimating principal eigenvalues. It 
consists of a min-max formula for the ratio  $-\Delta_Vu/u$, 
taken over all functions $u$ on a positive cone of $H^1_0(M)$.  

In Theorem \ref{integralminmax}, we describe  a Rayleigh type variational 
principle for $\lambda^*_V$ on a regular coordinate chart $\bar{M}$. This 
variational principle was initially due to Holland \cite{Ho} for a certain type 
of second-order linear elliptic equations on domains of Euclidean space. 
Later, it was reformulated by Godoy, Gossez and Paczka in \cite{GGP}, using suitable 
weighted Sobolev spaces. This provided an alternative proof of the formula.
Instead of Holland's method which uses ergotic measures to obtain a positive solution,
 $G_V$, of a related degenerate elliptic second order differential equation, 
it consits of applying Krein-Rutman theory for compact, positive and irreducible 
operators on  weighted $L^2$ spaces. It thereby requires less regularity 
conditions on the domain and coefficients of the operator. We follow this second 
approach, taking weighted Sobolev spaces on $\bar{M}$  weighted  by the square of 
the intrinsic distance function to $\partial M$, 
$d_{\partial M}(p)=\inf_{x\in \partial M}d(p,x)$, for $p\in \bar{M}$, defined 
in (\ref{sobolev1})-(\ref{sobolev2}). In Lemma \ref{omegaVboundary}(3) we 
obtain Sobolev embedding theorems 
in case $\bar{M}$ is a global chart domain, generalizing the known Euclidean case.
The variational  principle is given by 
$$\lambda_V^*=\inf_{\{u\in \mathcal{D}_{\partial}: ~\|u\|_{L^2}=1\}}\LA{(} \mathcal{L}(u,u)-
\inf_{v\in H^1_{\partial}(M)}{Q}_u(v)\LA{)},$$
where
$$\left\{\begin{array}{lcl}
\mathcal{L}(u,u)&=&\int_M (|\nabla u|^2+u\,g(V,\nabla u))dM,\\
Q_u(v) &=& \int_Mu^2(|\nabla v|^2-g(V,\nabla v))dM.\end{array}\right.
$$
Here, the infimum is taken over all  $L^2$-unit functions  of the class\\[-3mm]
\begin{equation}\label{Dd}
 \mathcal{D}_{\partial}:=\LA{\{}u\in H^1(M): \frac{u(p)}{d_{\partial M}(p)}\in 
[C_1,C_2] \mbox{~for~some~constants~}C_i>0\LA{ \}}.
\end{equation}
This infimum is achieved at a function $u_V\in \mathcal{D}_{\partial}$, $L^2$-normalized,
and given by the product $u_V=\omega_V\cdot \sqrt{G_V}$, where $G_V$ is a bounded, positive,
weak  solution of the degenerate elliptic differential equation,  
$\mathrm{div}^0(\omega_V^2(\nabla G+GV))=0$, that is, a solution of the integral 
equation (\ref{GV}). Furthermore, it is unique in the weighted Sobolev 
space (\ref{sobolev2}), up to a  multiplicative constant. In Proposition \ref{Proposition5.4},
we show that,  when $V$ is the gradient of a function $f$, it turns out that 
$G_V= e^{-f}$, and  this variational principle reduces to the Rayleigh 
vartiational principle  for the first  eigenvalue $\lambda_f$, given in 
equation (\ref{Rayleigh f}).

These formulas allow us to obtain comparison results for the principal 
eigenvalue on geodesic balls,  under pointwise comparison of the radial curvatures 
 and  the radial component of $V$, with the ones of model spaces.
On $N$, the radial direction from a point $p_0$ is defined  by $\partial_t(p)=\nabla r(p)$, where $t=r(p)=d(p, p_0)$ is the intrinsic distance of $p$ to $p_0$.
The exponential map of $N$ defines the spherical geodesic parametrization of a closed geodesic  ball $\bar{M}=\bar{B}_{r_0}(p_0)$. Namely,   
 $p=\exp_{p_0}(t\xi)=:\hat{\Theta}(t,\xi)$, with  $\xi$ in the unit sphere of $T_{p_0}N$,
 and $0\leq  t\leq r_0$. The radial component of $V$ is given by its projection
onto the radial direction,  $h_1(t,\xi)=g(V(p),\partial_t(p))$.
Our model spaces are geodesic balls $\bar{M}^{\rho}$ of  spherically symmetric spaces, 
$N^{\rho}=[0,l)\times_{\rho}\mathbb{S}^{m-1}$, endowed with a radial vector field,
$V^{\rho}=h(t)\partial_t$. The warping function $\rho$ is chosen based on  pointwise comparison of the radial curvatures with the ones of $N$. The function $h(t)$ is
chosen based on pointwise comparison with $h_1(t,\xi)$ of $V$.
Comparison theorems for the 
first eigenvalue of $-\Delta_0$ on a geodesic ball were obtained by Cheng in 
\cite{Cheng}, using space forms as model spaces. These theorems were  generalized by  Freitas, Mao and the second author in \cite{FMS}, taking as model spaces  the larger class of spherically symmetric spaces.

Next, we  state our two main theorems on a closed geodesic ball $\bar{M}
=\bar{B}_{r_0}(p_0)$, endowed  with a vector field $V$. 
 We are assuming  $r_0<\min\{\mathrm{inj}(p_0), l\}$, where $\mathrm{inj}(p_0)$ 
is the injectivity  radius of $p_0\in N$. 
The ball $M^{\rho}$ in the model space is centered at the origin and has  
radius $r_0$. The radial sectional 
curvature of $N^{\rho}$ is given by $-\rho''(t)/\rho(t)$. This curvature 
is a constant $\kappa$ in the case of space forms. Namely, 
the spheres when $\rho(t)= (\sqrt{\kappa})^{-1}\sin \sqrt{\kappa}t$, for 
$\kappa>0$, the Euclidean space when $\rho(t)=t$, for $\kappa=0$,
and the hyperbolic spaces when $\rho(t)= (\sqrt{-\kappa})^{-1}\sinh 
\sqrt{-\kappa}t$,  for  $\kappa <0$.
The radial vector field on the model space, $V^{\rho}$, depends on $t$ only, 
with initial condition $h(0)=0$. The principal eigenvalue is the first 
eigenvalue $\lambda_{\rho, H}$ of the Bakry-\'{E}mery $H$-Laplacian 
$\Delta^{\rho}_H$ on $\bar{M}^{\rho}$, with  Dirichlet boundary condition, 
where $H'(t)=h(t)$. In Section 4, we describe properties of the 
corresponding  principal eigenfunction $\omega_{\rho,H}$. We also describe
the whole spectrum of $-\Delta_{\rho, H}$, relating to a family of
 one-dimensional eigenvalue problems and the spectrum of the $(m-1)$-sphere. 

\begin{theorem} \label{Theorem 1.1} 
We assume the radial sectional curvatures of $M$, $K(\partial_t,X)$, 
and the radial component of $V$ satisfy at  each  point $p=\hat{\Theta}(t,\xi)$,
\begin{eqnarray}
 K(\partial_t,X) &\leq& -\frac{\rho''(t)}{\rho(t)},\label{radialSC}\\[-1mm]
 h_1(t,\xi) &\leq&  h(t), 
\end{eqnarray}
for all $ 0\leq t\leq r_0$, and unit vectors 
$\xi\in T_{p_0}M$ and  $X\in T_pM$ orthogonal to $\partial_t(p)$.
Then, we have,  $\lambda_V^*\geq \lambda_{\rho, H}$. Furthermore,
equality of the eigenvalues holds if and only if $M$ is isometric to 
$M^{\rho}$ and $h_1(t,\xi)=h(t)$, for all $(t,\xi)$. In this case, the principal 
eigenfunctions are the same, that is, $\omega_V=\omega_{\rho,H}$. 
\end{theorem}
\noindent
If $V=0$ and $H=0$, this is Theorem 4.4 of \cite{FMS}. 
Applying  Theorem \ref{Theorem 1.1} to vector fields $V$ on a geodesic ball 
 of a model space  $N^{\rho}$, we conclude that the principal eigenvalue 
$\lambda_{V}^*$ is just $\lambda_{\rho, H}$, if the radial component of $V$ 
depends on $t$ only. Therefore, in this case, the principal 
eigenvalue does not depend on the non-radial component of $V$, and the principal 
eigenfunction $\omega_V$ of the $V$-drift Laplacian is the radial first  
eigenfunction $\omega_{\rho,H}$ for the $H$-Laplacian. 

\begin{theorem}\label{Theorem 1.2} 
We are given  radial vector fields, $~V(p)=h_1(t,\xi)\partial_t$ on $\bar{M}$, 
and $V^{\rho}=h(t)\partial_t$ on  $\bar{M}^{\rho}$. We assume that 
$h(t)\geq 0$, and $~h_1(0,\xi)=h(0)=0$ holds  for all unit vectors $\xi\in T_{p_0}M$. 
We also assume that the radial Ricci curvatures of $M$ and $V$ 
satisfy the following inequalities 
\begin{eqnarray}
 \mathrm{Ricci}(\partial_t,\partial_t) &\geq& -(m-1)\frac{\rho''(t)}{\rho(t)}, 
\label{radialRicci}\\[-1mm]
 \mathrm{div}^0(V)(t,\xi)-\frac{1}{2}|V|^2(t,\xi)&\geq&
 \mathrm{div}^0_{\rho}(V^{\rho})(t)
-\frac{1}{2}|V^{\rho}|^2(t), \label{extra} 
\end{eqnarray}
for all $t,\xi$, with $t\leq r_0$. Then $\lambda_V^*\leq \lambda_{\rho,H}$,
and  equality of the eigenvalues holds if and 
only if $M$ is isometric to $M^{\rho}$ and equality holds in (\ref{extra}), 
for all $0\leq t\leq r_0$.  In this case, the principal eigenfunctions are related 
by the formula,  
$\omega_V(t,\xi)=\omega_{\rho,H}(t)e^{\frac{-H(t)+H_1(t,\xi)}{2}}$,  
where $\frac{dH_1}{dt}(t,\xi)=h_1(t,\xi)$,  and $H'(t)=h(t)$.  
If, additionally, $\rho(t), h(t)$, and $ h_1(t,\xi)$ are analytic 
on  $t\in [0, r_0]$, then  $h_1=h$ and $\omega_V=\omega_{\rho,H}$ .
\end{theorem}
The above theorem, in  case $V=V^{\rho}=0$  coincides with
Theorem 3.6 of \cite{FMS}. 
The assumption $r_0<\mathrm{inj}(p_0)$ can be dropped if the min-max 
formula in Theorem \ref{integralminmax} is valid on domains with less 
boundary regularity. Inequality 
(\ref{extra}) at $t=0$ means $h'_1(0,\xi)\geq h'(0)$, for all $ \xi$.

In \cite{HNR,HNR1, HNR2}, comparison results are obtained on an open domain 
$\Omega$ of Euclidean space, where the infimum and the supremum of $\lambda_V^*$ 
are searched among  all vector fields $V$ with $\|V\|\leq \tau$, for a fixed 
constant $\tau\geq 0$. The model space is the Euclidean disk with volume 
$|\Omega|$, endowed with the bounded radial vector field  $\tau\, x/|x|$, 
not defined at  $x=0$.  
In \cite{HNR3}, $\tau$ is  allowed to be a  radial  function $\tau(|x|)$, 
and  a  suitable symmetric rearrangement of the drift-Laplacian 
on the disk is taken. The results are  obtained  under 
comparison of  $L^{\infty}$ or $L^2$  norms of the vector field. Presently, 
we allow our model spaces to be geodesic disks of any spherically symmetric 
space,  endowed with any smooth radial  vector field $V(r(x))$,  vanishing at 
the origin. Our method is based on comparing pointwise the radial part of the 
vector fields and the radial curvatures. Radial curvature comparison conditions, 
as stated in the above  theorems, can be translated into comparison conditions
between volumes of geodesic balls of $N$ and $N^{\rho}$ of radius $t\leq r_0$.
Namely, (\ref{radialSC}) and (\ref{radialRicci})  correspond to
nondecreasing and nonincreasing ratio volume elements $\theta(t,\xi)$,
defined in (\ref{ratio}), respectively (see \cite{FMS}).

A simple application of the min-max formulas leads to some comparison results 
between $\lambda_0$ and $\lambda_V^*$  in Proposition \ref{simple}.  In
Corollary \ref{lambdav-lambda0} we conclude that, if $\mathrm{div}^0(V)\leq 0$,
then $\lambda_0\leq \lambda_V^*$.
In the particular case $V=\nabla f$, we get the following conclusion for a 
variation on the first eigenvalue $\lambda_f$.
\begin{proposition} If $f\in C^{\infty}(\bar{M})$ has constant $0$-Laplacian, 
$\Delta_0 f=2c_0$, then $\frac{d}{d\epsilon}|_{\epsilon=0}\lambda_{\epsilon f}$
exists and it is equal to $-c_0$.
\end{proposition}
 
We may question geometric properties of eigenvalues of $-\Delta_V$, 
real or complex; not only the principal eigenvalue. Another natural development
will be the study, in the Riemannian context,  of the variation of the principal 
eigenvalue for domain variations  under variational constraints.
An extension of the variational principle for $\lambda^*_V$ to any regular 
Riemannian domain could be obtained by extending to such domains the embedding results 
on weighted Sobolev spaces given  in Lemma \ref{omegaVboundary}(3), 
and main result of \cite{SZ}.

\section{The $V$-Laplacian}
 
We consider $\bar{M}=M\cup \partial M$ a smooth, compact domain with boundary, 
which is contained in a smooth $m$-dimensional Riemannian manifold $(N,g)$. 
Denote by $\nabla^0$ its Levi Civita connection. We will use the subscript or 
superscript $0$ on geometric objects that are defined with respect to $\nabla^0$. 
Given a smooth vector field $V$ on $\bar{M}$, we define a new connection by 
\begin{equation}\label{Eqn: vect-conn}
\nabla^V_X Y = \nabla^0_X Y + \sm{\frac{1}{m-1}}(g(X,Y) V - g(V, Y) X).
\end{equation}
This is a metric connection, i.e.
 $0=\nabla^V_Z g(X,Y)=Z\cdot(g(X,Y))-g(\nabla^V_ZX,Y)-g(X,\nabla_Z^VY)$.
The torsion is given by
$$
T(X,Y)=\sm{\frac{1}{m-1}}(g(V,X)Y-g(V,Y)X),$$ 
and it is one of the two distinguished types out of the three torsion types
for metric connections, namely the vectorial torsion as named by  Cartan
\cite{Cartan}.
For each  function  $u:M\to \mathbb{R}$ of class $C^2$, the Laplacian of $u$ 
with respect to the affine connection $\nabla^V$ is given by
\begin{equation}
\label{Laplacian}
\Delta_V u:= \mathrm{tr}(\nabla^V du)=\Delta_0u -g({V},\nabla u), 
\end{equation}
where $\nabla u$ is the $g$-gradient of $u$. This is the so-called
Laplacian with drift  the vector field $V$. In case $V$ is the gradient 
of a function $f$, this is the Bakry-\'{E}mery $f$-Laplacian,
 $$\Delta_fu=\Delta_0 u-g(\nabla f, \nabla u)=
e^{f}\mathrm{div}^0(e^{-f}\nabla u).$$
The $f$-Laplacian is self-adjoint for the $e^{-f}$-weighted $L^2$ space, 
$L^2_{e^{-f}}(M)$, that is
$$\int_M v\,\Delta_fu\,\, e^{-f}dM=\int_M u\,\Delta_fv\,\, e^{-f}dM,
\quad\quad\forall u,v\in C^{\infty}_c(M),$$
where $ C^{k}_c(M)$ is the space of functions of class $C^k$
($0\leq k\leq +\infty$) with compact support in the interior of $M$. 
Equivalently, $\Delta_f$ is $L^2$-self-adjoint for the conformally equivalent metric
$\hat{g}=e^{-\frac{2}{m}f}g$. Note that $\hat{\Delta}u=e^{\frac{2}{m}f}\Delta_V u$, 
where $V=-\frac{2(m-1)}{m}\nabla f$ is of gradient type.

\begin{lemma} \label{nonself} 
The $V$-Laplacian is $L_{e^{-f}}^2$-self-adjoint for some density function $e^{-f}$ 
if and only if $~V=\nabla f$. 
\end{lemma}
\noindent
\begin{proof} 
Let $u,v\in C^{2}_c(M)$. Applying Stokes's theorem to 
$\mathrm{div}^0(e^{-f}(u\nabla v-v\nabla u))$ we get,
\begin{equation}\label{symmetric}
\int_M(u\Delta_Vv-v\Delta_Vu)e^{-f}dM= 
\int_M g(\nabla f-V, u\nabla v-v\nabla u)e^{-f}dM.
\end{equation}
Assume $\Delta_V$ is $L_{e^{-f}}^2$-self-adjoint. Hence (\ref{symmetric}) $=0$
holds. Take  any $u\in C^{2}_c(M)$. Let $v\in  C^{2}_c(M)$ with $v=1$ on a 
neighbourhood of the support of $u$. The above equality implies 
$\int_M g(\nabla f-V,\nabla u)dM=0$. Thus, for any $u,v\in  C^{2}_c(M)$, 
we have $\int_M g(\nabla f-V,\nabla (vu))dM=0$. From 
$\nabla (uv)=u\nabla v + v\nabla u$ and (\ref{symmetric})$=0$ we obtain
$~\int_M 2v g(\nabla f-V,\nabla u)dM=0$. Since $v$ is arbitrary, then 
$g(\nabla f-V,\nabla u)=0$ for all $p\in M$, and $V=\nabla f$, necessarily.\qed 
\end{proof} 

The formal adjoint of  $\Delta_V:C^{2}_c(M)\to C^{0}_c(M)$ is the operator 
$\Delta^*_V:C^{2}_c(M)\to C^{0}_c(M)$, given by
\begin{equation}
\label{formaladjoint}
\Delta^*_Vv= \Delta_0 v+g(V,\nabla v) +\mathrm{div}^0(V)v.
\end{equation}
It can be extended as an operator defined on $H^1_0(M)$ (see notations in 
Section 3), which satisfies for all $u,v \in C^{2}_c(M)$
$$\int_M v\,\Delta_V u \,dM =\int_M u\,\Delta^*_V v\, dM.$$

If  $V=\nabla f$, for some  function $f\in C^{2}(\bar{M})$, it is known 
that the Dirichlet eigenvalue problem $\Delta_fu+\lambda u=0$, $u=0$ on 
$\partial M$, consists of a discrete sequence 
$0<\lambda_1<\lambda_2\leq \lambda_3\ldots \to +\infty$, \cite{LR}.
Furthermore, assuming each eigenvalue is repeated the number of times equal 
to its multiplicity, we may take  $\{\phi_1,\phi_2,\ldots\}$ a complete orthonormal 
basis of  $L^2_{e^{-f}}(M)$, composed of the corresponding
eigenfunctions. The first eigenvalue $\lambda_f:=\lambda_1$ 
is positive, of multiplicity one, and satisfies a Rayleigh variational principle:
\begin{equation}\label{Rayleigh f}
\lambda_f=\inf_{u\in H^1_0(M)}\frac{\int_M|\nabla u|^2e^{-f}dM}{
\int_M u^2e^{-f}dM}=\inf_{u\in C^{\infty}_c(M)} 
\frac{\int_M -u\Delta_fu\, e^{-f}dM}{
\int_M u^2e^{-f}dM}.   
\end{equation}
The infimum is achieved at $u$ if and only if $u$ is the $\lambda_f$-eigenfunction
$\omega_{f}$. All eigenvalues satisfy a similar variational principle (cf. \cite{LR}).

\section{The principal eigenvalue}

As in the previous section, we are assuming $\bar{M}=M \cup \partial M$ 
is a smooth compact domain  with  boundary contained in a complete 
Riemannian manifold $N$. We consider the Sobolev space
$$H^1(M)=\{u\in L^2(M): \exists \nabla u \in L^2(TM)\},$$
endowed with the $H^1$-norm
$\| u\|_{1}^2= \|u\|_{L^2}^2+ \|\nabla u\|_{L^2}^2$, 
where $\|u\|_{L^2}^2=\int_M u^2\, dM$ and 
$\|\nabla u\|^2_{L^2}=\int_M\|\nabla u\|_g^2 \,dM$.
Here, $\nabla u \in L^2(TM)$ means  the weak gradient of $u$
(cf.\ \cite{Cha}, I.5, Definition 4). The subspace 
$H^1_0(M)$ is the $H^1$-closure of  $C^{\infty}_c(M)$,
and $C^{\infty}_c(M)$ is $L^2$-dense on $L^2(M)$. We recall that  
$H^1_0(M)\cap C(\bar{M})\subset C_0(\bar{M})$, where $C(\bar{M})$
 is the space of continuous functions on $\bar{M}$ and
$C_0(\bar{M})$ its subspace of  functions  that 
vanish on $\partial M$. 
Conversely, $C^1(\bar{M})\cap C_0(\bar{M})\subset H^1_0(M)$
(\cite{Au}, 3.50 and Remark).
A more recent result (cf. \cite{SZ}), guarantees 
that if ${M}$ is diffeomorphic to a Euclidean domain with boundary of 
class $C^1$, then  $H^1(M)\cap C_0(\bar{M})\subset H^1_0(M)$. 

Let $H^1_0(M)'$ be the dual space of $H^1_0(M)$ of the bounded linear 
functionals on $H^1_0(M)$ with supremum norm, 
$\|F\|:=sup_{\|u\|_1=1} |F(u)|$, $\forall F\in H^1_0(M)'$. 
By the Riez representation theorem, $J:H^1_0(M)\to H^1_0(M)'$,
$J(u)(v):=(u,v)_1$, is a continuous isomorphic surjective isometry 
with continuous inverse $J^{-1}$. The  linear operator
$ I:L^2(M)\to H^1_0(M)'$,  $ If(v):=F_f(v):=\int_M fv\, dM$,
defines a continuous embedding of $L^2(M)$ into $H^1_0(M)'$ with 
$\|If\|\leq \|f\|_{L^2}$.  The usual embedding $H^1(M)\subset L^2(M)$ 
is compact for $m\geq 2$ (\cite{Au}, Kondrakov Theorem 2.34), inducing
 a compact operator $I :H^1_0(M)\to H^1_0(M)'$. 
We have $ \|F_f\|=\|u_f\|_1\leq \|f\|_{L^2}$, where $u_f=J^{-1}(F_f)$,
and  $(f,v)_{L^2}=F_f(v)=(u_f,v)_1$ for $v\in H^1_0(M)$.

For  each constant $\epsilon$, we consider the $H^1$-continuous 
bilinear functionals
$\mathcal{L}_{\epsilon},~ \mathcal{L}^*_{\epsilon}\, 
:H^1_0(M)\times H^1_0(M)\to \mathbb{R},$
\begin{eqnarray*}
\mathcal{L}_{\epsilon}(u,v) &= &\int_M \!g(\nabla u, \nabla v)dM+
\int_M\! g({V}, \nabla u )v\,dM + \epsilon \int_M \!\!uv \,dM, \\
\mathcal{L}^*_{\epsilon}(u,v) &=& \int_M \!\!g(\nabla u, \nabla v)dM
-\int_M \!\!g({V}, \nabla u )v\,dM
 +\int_M \!\!(\epsilon-\mathrm{div}^0({V}))uv\,dM. 
\end{eqnarray*}
The norm of a multilinear operator that is considered is the supremum norm,
$\|\mathcal{L}_{\epsilon}\|=$ $\sup_{\|u\|_1=\|v\|_1=1}
|\mathcal{L}_{\epsilon}(u,v)|$. We have 
$\mathcal{L}^*_{\epsilon}(u,v)=\mathcal{L}_{\epsilon}(v,u)$,
and continuous operators  ${L}_{\epsilon},:H^1_0(M)\to  H^1_0(M)'$,
and $\hat{L}_{\epsilon}: H^1_0(M)\to H^1_0(M)$ (similarly ${L}_{\epsilon}^*$ 
and $\hat{L}^*_{\epsilon}$), defined by
$\mathcal{L}_{\epsilon}(u,v)= {L}_{\epsilon}u(v)=(\hat{L}_{\epsilon}u, v)_1$.
We can naturally extend the operators $\mathcal{L}_{\epsilon}(u,v)$ and 
$\mathcal{L}_{\epsilon}^*(u,v)$, for $u\in H^1(M)$ and $v\in H^1_0(M)$.

Given continuous functions $f\in C(M)$ and $\phi\in C(\partial M)$, 
 \em a strong solution \em is a function $u\in C^2(M)\cap C(\bar{M})$ that
satisfies $-\Delta_V u +\epsilon u=f$ at every $p\in M$ and $u=\phi$ at 
any  $p\in \partial M$. Given $F\in H^1_0(M)'$, a \em weak solution \em 
of $-\Delta_V u+\epsilon u= F$,  $u=0$ on $\partial M$, is an element 
$ u\in H^1_0(M)$ that satisfies $\mathcal{L}_{\epsilon}(u,v)=F(v)$,  
for all $v\in H^1_0(M)$. We have 
$\hat{L}_{\epsilon}u=\hat{L}u + \epsilon J^{-1}Iu$,
${L}_{\epsilon}u={L}u+\epsilon Iu$, or in simplified notation, 
$L_{\epsilon}u=Lu+ \epsilon u=-\Delta_V u+\epsilon u$. 

Using  local  coordinate charts on $M$, we apply  theorems 
of Chapter 8 of \cite{GT},  Section 3.6 of \cite{Au}, or Section 6.3 
of \cite{Ev}, to determine regularity of solutions on open domains $M$.
If $u\in H^1_0(M)$ is a weak solution  of
 $-\Delta_V u+\epsilon u=f\in L^2(M)\subset H^1_0(M)'$, then $u\in H^2(M)$. 
Furthermore, under additional conditions,  we have the 
following conclusions.  If $f\in H^k(M)$, then 
$u\in H^{k+2}(M)$ and $\|u\|_{k+2}\leq C(\|u\|_{L^2}+ \|f\|_{k})$.
If  $f\in C^{k,\alpha}(M)$ (resp.\ $C^{\infty}(M))$,  
then $u\in C^{k+2,\alpha}(M)$ (resp. $C^{\infty}(M))$. 
For regularity up to the boundary we have the Sobolev theorem
on any smooth Riemannian manifolds with boundary (\cite{Au}, Theorem 2.30). 
Namely, if $u\in H^{k}(M)$ and $2k\geq n+2s$, where $s\geq 0$, 
then $u\in C^{2s}(\bar{M})$. In this case, since we are assuming 
$u\in H^1_0(M)$, $u=0$ on $\partial M$. Furthermore, if $s\geq 2$, 
then $u$ satisfies  $(-\Delta_V u+\epsilon u, \phi)_{L^2}=(f,\phi)_{L^2}$, 
$\forall \phi\in C_c^{\infty}(M)$. Hence, $u$ is a strong solution
of  $L_{\epsilon}u=f$. In particular, all weak solutions
$u\in H^1_0(M)$ of $-\Delta_V u+\epsilon u=0$ are in $C^{\infty}(\bar{M})$
and vanish on $\partial M$. 
The same conclusions hold for $\Delta_V^*-\epsilon$. If $\epsilon=0$ we omit 
the number in $L_{\epsilon}$.  There is uniqueness of weak solutions 
$u\in H^1_0(M)$ of $\Delta_Vu-\epsilon u =f$, for all $\epsilon\geq 0$. 
The same holds for $\Delta^*_V-\epsilon $ if $\mathrm{div}^0(V)\leq \epsilon$.
They are consequences of uniqueness of generalized Dirichlet problems
(cf. \cite{GT} Corollary 8.2 and Theorem 8.3, whose arguments are valid in any
regular Riemannian compact domain). Uniqueness results can be derived from 
maximum principles, and existence results from Fredholm theory for compact 
operators on Hilbert spaces. Both theories are used to prove the existence of a 
principal eigenvalue. We recall some maximum principles that we need 
(cf.\ \cite{Au}, Theorem 3.74).
\begin{theorem} \label{MaxPrinc} 
Assume $M$ compact with boundary, and let $\nu$ be the unit outer normal to 
$\partial M$. Let $u\in C^2(M)$ such that $\Delta_V u\geq 0$.\\[1mm]
$(a)$ (weak maximum principle) If $u\in C(\bar{M})$ and
$u_{|\partial M}\leq 0$, then $u\leq 0$.\\[1mm]
$(b)$ (Hopf maximum principle) If  $u$ achieves a nonegative maximum 
$T\geq 0$  at $p\in M$, then $u$ is constant.\\[1mm]
$(c)$ (boundary condition) Assume $u\in C(\bar{M})$ and  $u\leq 0$. 
If $u$ is not constant, and $u(p_1)=0$ at $p_1\in \partial M$, then 
$\frac{\partial u}{\partial \nu}(p_1)>0$, provided this derivative exists.\\[1mm]
The same holds for  $\Delta_V-\epsilon$ for any constant $\epsilon\geq 0$,
and for $\Delta^*_V-{\epsilon}$ if $\epsilon\geq 
\mathrm{div}^0({V})$.
\end{theorem}
We will see that the weak maximum principle holds for $\Delta^*_V$ in 
Proposition \ref{MaxPrinc*}.

Next we recall how the Fredholm alternative theorem describes the spectrum 
of $-\Delta_V$. If $V$ is not of gradient type, the set of eigenvalues
$\Lambda(-\Delta_V)$ may have complex numbers, with complex  eigenfunctions. Moreover,  a non self-adjoint $V$-Laplacian for any 
$L^2$-inner product does not have an $L^2$-diagonalization process that splits 
$L^2(M)$  into eigenspaces.  A long computation can show that, in general, 
$\Delta_V$ is not a normal operator unless $V$ is a parallel vector field.  

From the Rayleigh principle for $V=0$, we can take $\epsilon$ sufficiently 
large such that coerciveness of $\mathcal{L}_{\epsilon}$ is satisfied, that is, 
$\mathcal{L}_{\epsilon}(u,u)\geq \beta \|u\|_1^2$, for a positive constant 
$\beta$. Hence,   the bounded linear operators $L_{\epsilon}$ and 
$L_{\epsilon}^*$ are isomorphisms with bounded inverses. 
The continuous operator $T_{\epsilon}=L_{\epsilon}^{-1}\circ I:L^2(M)\to H^1_0(M)$  
satisfies $\mathcal{L}_{\epsilon}(T_{\epsilon}f,v)=(f,v)_{L^2}$, and
it is a compact operator as an operator on $L^2(M)$ (as well as on $H^1_0(M)$). 
Similarly, we define a compact operator  $H_{\epsilon}$ from 
${L}_{\epsilon}^*$. Then  we have, for any $u, v\in C^\infty_0(M)$,
\begin{eqnarray*}
(u,v)_{L^2} &=& Iu(v)=\mathcal{L}_{\epsilon}(T_{\epsilon}u,v)\, =\, 
\mathcal{L}^*_{\epsilon}(v, T_{\epsilon}u)\\
&=& \mathcal{L}^*_{\epsilon}(H_{\epsilon} H^{-1}_{\epsilon}v, T_{\epsilon}u)
\, =\, ( H^{-1}_{\epsilon}v, T_{\epsilon}u)_{L^2}\, =\, 
(T^*_{\epsilon}H^{-1}_{\epsilon}v,u)_{L^2},
\end{eqnarray*}
where $T^*_{\epsilon}$ is the adjoint operator of $T_{\epsilon}$
for the $L^2$-norm. Thus, $H_{\epsilon}$ is just $T^*_{\epsilon}$.
Consequently, $H_{\epsilon}$ and $T_{\epsilon}$ have the same  spectrum
($\mu=0$ included). Moreover, $u$ is an eigenfunction of $T_{\epsilon}$ 
for a nonzero eigenvalue $\mu$, i.e.\ $T_{\epsilon}u=\mu u$, if and only 
if $u$ is an eigenfunction of $-\Delta_V$ for the eigenvalue 
$\lambda= \mu^{-1}-\epsilon$. From the Fredholm theory applied to the 
compact operator $T_{\epsilon}$ and its adjoint on $L^2(M)$ (as in \cite{Ev})  
(we can also use  $H^1_0(M)$ as in \cite{GT}), 
the set  of eigenvalues of $T_{\epsilon}$ is the same of its adjoint. It is 
either a finite set or a sequence converging to zero, and the dimension of 
each eigenspace is finite. Among these eigenvalues there is a distinguished 
one, the principal eigenvalue, that can be described using Krein-Rutman 
theory. This theory only requires  $\partial M$ to be of class $C^{2,\alpha}$,  
and the metric $g$ and the vector field $V$  of class $C^{1,\alpha}$ on 
$\bar{M}$. One considers $T_{\epsilon}$ as a compact operator on 
$C^{1, \alpha}_0(\bar{M})$, $T_{\epsilon}:C^{1, \alpha}_0(\bar{M})
\to C^{1, \alpha}_0(\bar{M})$.
In this case  $K=\{v\in C^{1,\alpha}_0(\bar{M}): v\geq 0\}$
is a solid cone, whose interior is given by
$K^o=\{v\in K: v>0 ~\mbox{on}~M,\partial v/\partial\nu <0\}$.
This cone $K^o$ is not empty since $\partial M$ is of class $C^1$, 
as we can see from  Lemma \ref{omegaVboundary}. We can build a function 
$v\in K^o$  gluing a constant function $v_0$ with $d_{\partial M}$,
using a partition of unity.   
If $v\in K\backslash \{0\}$, applying the maximum principle with respect to
$\Delta_V-\epsilon$,  $u=T_{\epsilon}v\in K$, and by the Hopf maximum principle
we get $u\in K^o$. Thus, $T_{\epsilon}$ is strongly positive
with respect to $K$. Then, the Krein-Rutman theory states that 
the spectral radius $r(T_{\epsilon})$ of $T_{\epsilon}$ is a simple eigenvalue
of $T_{\epsilon}$, with eigenfunction $v_{\epsilon}\in K$.
Hence, $\omega_V:=r(T_{\epsilon})v_{\epsilon}=
T_{\epsilon}v_{\epsilon }\in K^o$ is a principal eigenfunction  of $-\Delta_V$ 
for the principal eigenvalue $\lambda_V= r(T_{\epsilon})^{-1}-\epsilon$, that is,
$$ \Delta_V \omega_V + \lambda_V \omega_V =0.$$
This means the pair  $(\lambda_V,
\omega_V)$  satisfies the following conditions (1)-(4). Moreover, if we choose
$\epsilon$ sufficiently large so that the maximum principle also holds
for $\Delta^*_V-\epsilon$, then we have a similar construction of a
pair $(\lambda^*_V,\omega_V^*)$ satisfying the same conditions: 
\begin{itemize}
\item[(1)] $\lambda_V$ is a simple eigenvalue, $\omega_V>0$ on $M$, $\omega_V=0$ on
$\partial M$.

\item[(2)] $\lambda_V<Re(\lambda)$, ~$\forall \lambda\neq \lambda_V$ (complex) 
eigenvalue of $\Delta_V$.

\item[(3)] No other eigenvalue  has a positive eigenfunction on $M$.

\item[(4)] $\partial \omega_V/\partial \nu <0$, that is,  $\omega_V\in K^o$.
\end{itemize}
\noindent 
It is known that (cf.\ \cite{BNV},  and \cite{Pad} for any compact 
Riemannian domain $\bar{M}$  with smooth boundary):
\begin{itemize}
\item[(5)] $\lambda_V>0$ if and only if the weak maximum principle holds.
\end{itemize}
This is the case of $\Delta_V$ as we stated in Theorem \ref{MaxPrinc}. 
On the other hand, we have equality of the  spectral radius 
$r(T^*_\epsilon)=r(T_{\epsilon})$ and we may conclude that 
$\lambda^*_V=\lambda_V$, and so $\lambda^*_V$ is also positive. 
Note that $\omega_V, \omega_V^*\in C^{1,\alpha}_0(\bar{M})\subset H^1_0(M)$,
and so they are in $C^{\infty}(M)\cap C_0(\bar{M})$ as well.  
Furthermore, applying the Fredholm alternative theorem to $T_{\epsilon}$ 
(\cite{GT}, Theorem 5.11), we obtain, in the following proposition, 
a description of the eigenvalues of $\Delta_V$ and  $\Delta_V^*$, 
as  operators on $L^2(M)$.

\begin{proposition}\label{MaxPrinc*} 
The Laplacians $-\Delta_V$ and $-\Delta^*_V$ have the 
same set $\Lambda$ of  eigenvalues,  a discrete set that can be either
a finite set or  a sequence $|\lambda_k|\to +\infty$.
The corresponding eigenspaces are finite dimensional subspaces of 
$ L^2(M)$. The principal eigenfunctions $\omega_V$ and $\omega_V^*$ 
lie in $ C_0^{1,\alpha}(\bar{M})\subset H^1_0(M)$,
and are smooth on $M$, and vanish on $\partial M$.
The weak maximum principle also holds for $\Delta_V^*$, 
independently of $\mathrm{div}^0(V)\leq 0$ holding or not.
\end{proposition}
The common positive principal  eigenvalue of $-\Delta_V$ and of $-\Delta_V^*$
 will be denoted by $\lambda_V^*$.

\section{Model spaces}

A spherically symmetric space is a warped product space, 
$N^{\rho}=[0, l)\times_{\rho} \mathbb{S}^{m-1}$,
endowed with the  warped  metric,
$$g_{\rho}= dt^2 + \rho^2(t)d\sigma^2,$$
where $\rho\in C^{\infty}([0,l))$ satisfies $\rho>0$ on $(0,l)$,   
$\rho(0)=\rho''(0)=0$, and $\rho'(0)=1$. Here, $d\sigma^2$ denotes the 
usual metric on the unit $(m-1)$-sphere. The origin of $N^{\rho}$ 
is the point $p_{\rho}$ defined by identifying all pairs $(0,\xi)$, where
$\xi \in \mathbb{S}^{m-1}$. The metric is smooth away from the origin,
and smooth at $p_{\rho}$ if we assume all even derivatives of $\rho$ 
vanish at $t=0$. The  distance function to $p_{\rho}$ is given by $r(t,\xi)=t$. 
Hence, $t\partial_t=\frac{1}{2}\nabla r^2$ is a smooth vector field.
We consider closed geodesic balls $\bar{M}^{\rho}:=\bar{B}_{r_0}(p_{\rho})$, 
centered at $p_{\rho}$ and of radius $r_0<l$. The radial sectional  curvature, 
and the radial Ricci curvature of $N^{\rho}$ are defined  at each point
$p=(t,\xi)$, (cf.\  \cite{FMS}) by
$$K^{\rho}(\partial_t, W)=-\frac{\rho''(t)}{\rho(t)},
\quad \mathrm{Ricci}^{\rho}
(\partial_t, \partial_t)=-(m-1)\frac{\rho''(t)}{\rho(t)},$$
\noindent
respectively, where  $W\in T_{\xi}\mathbb{S}^{m-1}$ has unit $g_{\rho}$-norm. 

A function $F(t,\xi)$ is said radial if it only depends on $t$,
that is, $F(t,\xi)=F(t)$. We will consider 
vector fields in the radial direction, depending on $t$ only, that is,
$ {V}^{\rho} = h(t)\partial_t=\nabla H$, 
where  $H\in C^{\infty}([0, r_0])$,  and $ h(t)=H'(t)$. 
We are always assuming that $h(t)$ is of the form  $h(t)=t\tilde{h}(t)$ for 
$t$ near $0$, and for some smooth function $\tilde{h}$. Hence, $h(0)=0$ and 
$V^{\rho}=\tilde{h}(t)(t\partial_t)$ is smooth on $[0, r_0]$.
This vector field is of gradient type, and so it defines a Bakry-\'{E}mery 
model space $(N^{\rho},g_{\rho}, e^{-H}dV)$.  We will denote the 
$V^{\rho}$-Laplacian $\Delta_{V_\rho}$, by the $H$-Laplacian 
$\Delta^{\rho}_H$. 
Fixing a $g_{\rho}$-orthonormal basis $\partial_t, e_i$, with 
 $e_i\in T_{\xi}\mathbb{S}^{n-1}$, for $1\leq i\leq  m-1$, we have
$$\begin{array}{l}
\nabla^0_{\partial_t}{{V}^{\rho}}=h'(t)\partial_t,~~
\nabla^0_{e_i}{{V}^{\rho}}=\rho' (t)\frac{h (t)}{\rho(t)}e_i,\\
\mathrm{div}^0({V}^{\rho})=h' (t) +(m-1)\frac{h(t)}{\rho(t)}\rho'(t).
\end{array}$$
Thus,   $\lim_{t\to 0^+}\mathrm{div}^0({V}^{\rho})(t)=m h'(0)$.
Consider the function 
\begin{equation}
 p(t)=\rho^m(t)e^{-H(t)}, \quad \mbox{where}~~ H(t)=h'(t).
\end{equation}
The Laplacian for the Levi-Civita connection
 has the following expression (cf. \cite{FMS}),
$$\Delta^{\rho}_0u= \frac{d^2u}{dt^2} + (m-1)\frac{\rho'}{\rho}\frac{du}{dt}
+\frac{1}{\rho^2}\Delta_{\mathbb{S}^{m-1}}u.$$
Hence, the $H$-Laplacian of a function $u(t,\xi)$ is given by,
\begin{equation}\label{H-Laplacian}
\Delta_{H}^{\rho} u=\frac{d^2u}{dt^2} + \frac{p'}{p}
\frac{d u}{dt}+\frac{1}{\rho^2}\Delta_{\mathbb{S}^{m-1}}u.
\end{equation}
In the following proposition we describe the properties of the principal 
eigenvalue $\omega_{\rho,H}$ of the $H$-Laplacian 
on $M^{\rho}$, with Dirichlet boundary conditions:
\begin{proposition} \label{radialmodel} Let $V^{\rho}=h(t)\partial_t$, 
with $h(0)=0$. Then, $\omega_{\rho,H}$ is  radial, and for each 
$ t\in (0, r_0)$ it satisfies
\begin{eqnarray}\label{principaleigenmodel}
\omega_{\rho,H}''(t) +\frac{p'(t)}{p(t)}\omega_{\rho, H}' (t)
+\lambda_{\rho,H}\omega_{\rho,H}(t)=0.
\end{eqnarray}
Furthermore,  $\omega_{\rho,H}(t)>0$ and $\omega_{\rho,H}'(t)<0$ on 
$(0,r_0)$, $\omega_{\rho,H}(r_0)=0=\omega'_{\rho,H}(0)$, and
 $\omega_{\rho,H}'(r_0)<0$.
\noindent
\end{proposition}
\noindent
\begin{proof} In this proof we will denote $\mathbb{S}^{m-1}$ by $S$. Let 
$0<\lambda_1< \lambda_2\leq \ldots \leq \lambda_l\leq \ldots \to +\infty$
be the set of eigenvalues of $\Delta_H^{\rho}$ on $M^{\rho}$ with
 Dirichlet boundary conditions. Let $u(t,\xi)\in C_0^{\infty}(M^{\rho})$ 
be an eigenfunction for one of the eigenvalues $\lambda=\lambda_l$. 
As in \cite{Cha}, p.\ 40-43 (for $H=0$),  we will decompose $u(t,\xi)$ in a sum of 
products of an eigenfunction of a one-dimensional eigenvalue problem, with 
a homogeneous harmonic polynomial. Let $\nu_k=k(k+m-1)$, $k=0,1,\ldots$, 
be the eigenvalues of the Laplacian $-\Delta_{{S}}$ on the $(m-1)$-sphere 
for the  closed eigenvalue problem. There is a complete orthornormal system of 
eigenfunctions $G_{k,\alpha}$ of $-\Delta_{{S}}$, defining a basis of 
$L^2(\mathbb{S}^{m-1})$. The index $k$ corresponds to the eigenvalue $\nu_k$, 
and $\alpha$ runs from  $1$ to $N_k$, the multiplicity of $\nu_k$.  
If we fix $t$, then   
\begin{equation}\label{representation}
u(t,\xi)=\sum_k\sum_{\alpha} a_{k,\alpha}(t)G_{k,\alpha}(\xi),
\end{equation}
for some constants $a_{k,\alpha}(t)=(u(t,\cdot),G_{k,\alpha})_{L^2(S)}$. 
We have
$$  \|u^2(t,\cdot)\|_{L^2(S)}^2 = \sum_{k,\alpha}a_{k,\alpha}^2(t),\quad
a_{k,\alpha}^{(s)}(t)  = \left(\frac{d^su}{dt^s}(t,\cdot)\,,\, G_{k,\alpha}\right)_{L^2(S)},
 ~s=0,1,2,\ldots.$$
Since $\nu_kG_{k,\alpha}=-\Delta_SG_{k,\alpha}$, using (\ref{H-Laplacian})
and $\Delta^{\rho}_Hu=-\lambda u$, we have
\begin{eqnarray*}
\nu_ka_{k,\alpha}(t) &=&-(u(t,\cdot ), \Delta_SG_{k,\alpha})_{L^2(S)}=
-(\Delta_Su(t,\cdot),G_{k,\alpha})_{L^2(S)}\\
&=&\rho^2\LA{(}-\lambda u+ \sum_{l,\beta}a''_{l,\beta}(t)G_{l,\beta}
+\sum_{l,\beta}\frac{p'(t)}{p(t)}a'_{l,\beta}(t)G_{l,\beta}
~,~G_{k,\alpha}\LA{)}_{L^2(S)}\\
&=& \rho^2\La{[}-\lambda a_{k,\alpha}(t)-
 \frac{1}{p(t)}(p a'_{k,\alpha} )'(t)\La{]}.
\end{eqnarray*}
Hence,  for each $k$,  if $a_{k,\alpha}\neq 0$ for some $\alpha$,  then $\lambda=\lambda_l$   
is a solution of the one-dimensional eigenvalue problem
\begin{equation}
 (pa')'(t) +(\lambda -\rho^{-2}\nu_k)p(t)a(t)=0, \label{eigenModel1}
\end{equation}
or equivalently,  a solution of
\begin{equation}
a''(t)+\left((m-1)\frac{\rho(t)}{\rho(t)}-h(t)\right)a'(t)+(\lambda-\nu_k\rho^{-2})a(t)=0. 
\label{eigenModel2}
\end{equation}
From the Dirichlet boundary conditions on $u$, and smoothness of
$u$ at $t=0$,  we must impose the following boundary conditions on $a(t)$,  
\begin{equation}\label{bc}
a(r_0)=0 = a' (0).
\end{equation}
Therefore, we conclude that each eigenvalue $\lambda_l$ of the $H$-Laplacian 
arises as a solution of at least one of the eigenvalue problems 
(\ref{eigenModel2}), with boundary condition (\ref{bc}), with respect to 
some $k$. Moreover, each $a_{k,\alpha}$ lies in the eigenspace $E_{k,\lambda}$ 
of the eigenvalue problem (\ref{eigenModel2}), when $k$ is fixed.

Reciprocally, let us  we fix an eigenfunction on the $(m-1)$-sphere, $G_{k}(\xi)$, 
with eigenvalue $\nu_k$, and an eigenfunction $a_k(t)$ of the eigenvalue 
problem  (\ref{eigenModel2})--(\ref{bc}), with respect to $\nu_k$ and  with 
eigenvalue $\lambda$.  Set   $u(t,\xi):= a_k(t)G_{k}(\xi)$. It satisfies
$u=0$ on $\partial M^{\rho}$, and for any $0<t<r_0$ and $\xi \in S$, we have
$$\Delta^{\rho}_H u(t,\xi)+\lambda u(t,\xi)
= -\rho^{-2}a_k(t)\nu_k G_{k}(\xi) +
 \frac{a_k(t)}{\rho^2(t)}\nu_kG_{k}(\xi)=0.
$$
Thus,  $u(t,\xi)$ is an eigenfunction of $-\Delta^{\rho}_H$ with eigenvalue 
$\lambda$.
In particular, for each $k$, the set of eigenvalues of the
one-dimensional  eigenvalue problem  (\ref{eigenModel2})--(\ref{bc})  
is a subset  $\{\lambda_{k,1},\ldots \lambda_{k,2}\ldots\}$ of 
$\{\lambda_1,\lambda_2,\ldots\}$.

Now we consider $k$ fixed, and  two solutions of  (\ref{eigenModel2}),  
$a_{k, i}$, $a_{k,j}$,  with eigenvalues $\lambda_{k,i}$ and $\lambda_{k,j}$, 
respectively. Then
$\int_{0}^{r_0} (pa'_{k,i})' a_{k,j}dt =-\int_0^{r_0} (pa'_{k,j}a'_{k,i})dt~$.
On the other hand,
$$
\int_{0}^{r_0} (pa'_{k,i})' a_{k,j}dt= -\lambda_{k,i}\int _0^{r_0}
p\,a_{k,i}a_{k,j}dt + \nu_k\int\rho^{-2}pa_{k,i}a_{k,j}dt.
$$
Hence, $(\lambda_{k,j}-\lambda_{k,i})\int_0^{r_0} p\,a_{k,i}a_{k,j}\, dt =0$.
That is, if $i\neq j$,  $a_{k,i}$ and $a_{k,j}$ are $L^2_{p(t)}$ orthogonal
on $[0,r_0]$.
For $i=j$ we have, 
$$\int_0^{r_0} p (a'_{k,i})^2dt = \lambda_{k,i}\int _0^{r_0}
 p\,(a_{k,i})^2dt - \nu_k\int_{0}^{r_0}\rho^{-2}p\, (a_{k,i})^2\, dt.$$
We also note that, if $k=0$, (\ref{eigenModel1}) gives 
\begin{equation}
p(t)a_{0,i}' (t)=-\lambda_{0,i}\int_0^tp(\tau)a_{0,i}(\tau)d\tau. \label{Maisuma}
\end{equation}
Similarly, if $a_k(t)$ and $a_s(t)$ are solutions of (\ref{eigenModel2})
for the same eigenvalue $\lambda$, associated with $\nu_k$ and $\nu_s$, 
respectively, then $(\nu_k-\nu_s)\int_0^{r_0}p\rho^{-2}a_ka_sdt=0$.
Hence, for $k\neq s$, $a_{k}$ and $a_{s}$ are $L^2_{p\rho^{-2}}$-orthogonal
on $[0,r_0]$.
Now, the volume element of $(N^{\rho}, g_{\rho})$ is given by  
$\rho^{m-1}dt\wedge dS$, where $dS$  is the volume element of $\mathbb{S}^{m-1}$. 
Thus,  the norm of the eigenfunction $u(t,\xi)$ on
the $e^{-H}$-weighted $L^2$-space on $M^{\rho}$ is given by
\begin{eqnarray*}
\|u\|_{L^2_{e^{-H}}(M^{\rho})}^2&=&\int_{0}^{r_0} \rho^{m-1}(t)e^{-H(t)}
\left(\int_{S}u^2(t,\xi)dS \right)dt\nonumber \\
&=&\sum_{k,\alpha}  \int_0^{r_0}p(t)a_{k,\alpha}^2(t)dt
= \sum_{k,\alpha}\|a_{k,\alpha}\|_{L^2_{p(t)}}^2. 
\end{eqnarray*}
The arguments given in \cite{Cha}, pp.\ 41, are valid 
concerning the eigenvalue problem   (\ref{eigenModel2}),
since $p(t)=\rho(t)^{m-1}e^{-H(t)}$ is qualitatively the same as $H=0$ 
(see also \cite{Zett}, p.\ 209). Thus, for each $k=0,1,\ldots$,
the solutions of the eigenvalue problem 
(\ref{eigenModel2}),   with boundary condition (\ref{bc}), 
consists of an  increasing sequence $\lambda_{k,i}$,  converging to infinity 
when $i\to +\infty$. Each $\lambda_{k,i}$ is simple, 
and the corresponding eigenfunction $a_{k,i}$ ($L^2_{p(t)}$-normalized) 
has $i-1$ zeros on $(0,r_0)$. 
Now,  $u$ in (\ref{representation}) is a $-\Delta_H^{\rho}$-eigenfunction 
for at least one of the eigenvalues $\lambda_{k,i}=\lambda_l$. Then, for  each $\alpha$,
$a_{k,\alpha}(t)= c_{k,i,\alpha}a_{k,i}(t)$, for some constants $c_{k,i,
\alpha}$. Consequently,  we have the following representation of a 
$\lambda_l$-eigenfunction $u$
as a $L^2_{e^{-H}}(M^{\rho})$-convergent series 
$$ u(t,\xi)=\!\!\!\!\!
\sum_{ \{k, i:\, \lambda_{k,i}=\lambda_l\}}\!\!\!\!\!\!
a_{k,i}(t)G_{k,i}(\xi), \quad\mbox{where~}~G_{k,i}(\xi)=
\!\sum_{\alpha} c_{k,i,\alpha}G_{k,\alpha}(\xi). $$
Moreover,  $~~\sum_{\alpha}\|a_{k,\alpha}\|^2_{L^2_{p(t)}}
\!\!\!=\sum_{\alpha}c_{k,i,\alpha}^2\!\!
=\|G_{k,i}\|_{L^2(S)}^2$, ~~and so
 $$\quad\|u|^2_{L^2_{e^{-H}}(M^{\rho})}=\!\!
\sum_{\{k,i:\lambda_{k,i}=\lambda_l\}}\!\!
\|G_{k,i}\|^2_{L^2(S)}=\!\!
\sum_{\{ k,i:\lambda_{k,i}=\lambda_l, \alpha \}}\!\!
c_{k,i,\alpha}^2.$$
Note that the only eigenfunction $G_{k,\alpha}$ that 
does not change of sign in $S$ is the constant function $G_{0,1}=1$.  
Now, the principal eigenvalue of $-\Delta^{\rho}_H$, is the lowest 
eigenvalue  $\lambda_{\rho,H }$, and the first eigenfunction,  
$\omega_{\rho,H}$, is positive  on $M^{\rho}$, 
only vanishing along the boundary. 
Hence,  $\omega_{\rho,H}(t,\cdot)$ corresponds to  the lowest eigenvalue of 
(\ref{eigenModel2}) with $k=0$. Consquently, $\lambda_{\rho,H}=\lambda_{0,1}$, 
and $\omega_{\rho,H}(t,\xi)=a_{1,0}(t)$, 
up to a multipicative positive constant. It is radial, positive for 
$t\in [0, r_0)$,   vanishes at $t=r_0$, and satisfies (\ref{eigenModel2}),
and thus,  (\ref{principaleigenmodel}).
 From (\ref{Maisuma}), and since $a_{0,1}(s)>0$, we conclude that the sign 
of $a'_{0,1}(t)$ is the same  of $-\lambda_1$, and $a' _{0,1}(r_0)<0$. This
completes the proof.  \qed
\end{proof}

In the above proof, we also have obtained the following conclusions.
\begin{proposition}
(1) For each $k$ fixed, the one-dimensional eigenvalue problem  (\ref{eigenModel2}) with 
boundary condition (\ref{bc}),  consists of an increasing sequence of simple 
eigenvalues, $0<\lambda_{k,1}< \lambda_{k,2}< \ldots \to +\infty$. 
Furthermore, eigenfunctions  $a_{k, i}$ and $a_{k,j}$, with respect to 
different eigenvalues $\lambda_{k,i}$ and $\lambda_{k,j}$, are 
$L^2_{p(t)}$-orthogonal. This means  
$ (a_{k,i},a_{k,j})_{L^2_p(0,r_0)}:=\int_0^{r_0} a_{k, i}a_{k,j} p \,dt=0$, for all $i\neq j$.
Moreover, if $k=0$,  $\|a_{0,i}' \|^2_{L^2_p(0,r_0)}=
\lambda_{0,i}\|a_{0,i} \|^2_{L^2_p(0,r_0)}$.\\[1mm]
(2) The discrete set of eigenvalues $\lambda_l\to +\infty$ of $-\Delta^{\rho}_H$
on the ball of radius $r_0$, $M^{\rho}$, with Dirichlet boundary condition, 
consists of the set $\{ \lambda_{k,i},  k=0,1,2, \ldots, i=1,2,\ldots\}$. 
Furthermore, each function $F\in L^2_{e^{-H}}(M^{\rho})$
can be expressed as an $L^2_{e^{-H}}(M^{\rho})$-convergent sum of
$-\Delta^{\rho}_H$-eigenfunctions $u_l(t,\xi)$, and so, 
 as a convergent sum  whose terms consists of  products of an  eigenfunction of
(\ref{eigenModel2}) with an eigenfunction of $-\Delta_{\mathbb{S}^{m-1}}$.
\end{proposition}

\section{Min-max formulas for the principal eigenvalue}

We first extend to the Riemannian case a min-max formula obtained by Protter-Winberger \cite{PW} in 1966 (see also \cite{BNV}), for open regular domains of 
$\mathbb{R}^n$. 
We assume that $\bar{M}$ is a compact regular domain of a complete Riemannian
manifold $(N,g)$.
 We consider the following cone of $H^1_0(M)$,
\begin{equation}\label{D+}
 \mathcal{D}_0^+=\{u\in C^2(M)\cap C(\bar{M})\cap H^1_0(M): 
u>0 ~\mbox{on}~M, u_{\partial M}=0 \}.
\end{equation}
\begin{theorem}[Min-Max formula] \label{Barta11}
The following min-max formula holds for the principal eigenvalue of 
$\Delta_V$,
$$\lambda^*_V=\sup_{u\in \mathcal{D}^+_0} 
\inf_{p\in M} -\frac{\Delta_V u(p)}{u(p)}= \inf_{u\in \mathcal{D}^+_0} 
\sup_{p\in M} -\frac{\Delta_V u(p)}{u(p)}.$$
We have the same formula with respect to $\Delta_V^*$.
\end{theorem}
\noindent
\begin{proof} 
We take $\omega_V, \omega_V^*\in \mathcal{D}^+_0\cap C^{\infty}(\bar{M})$, 
principal eigenfunctions of $-\Delta_V$ and $-\Delta^*_V$, respectively.
Then, 
$$
\lambda_V^*=-\frac{\Delta_V \omega_V}{\omega_V}
=\inf_M \,-\frac{\Delta_V \omega_V}{\omega_V}\leq 
\sup_{u\in \mathcal{D}^+_0}
\inf_M \,-\frac{\Delta_V u}{u}.$$
On the other hand, for any $u\in \mathcal{D}^+_0$, 
\begin{eqnarray*}
\inf_M \LA{(}-\frac{\Delta_V u}{u}\LA{)}\int_M u \omega_V^*dM &\leq& 
 \int_M -\frac{\Delta_V u}{u}u \omega_V^*dM
=-\int_M (\Delta_V u)\omega_V^*dM\\
&=&-\int_M u(\Delta^*_V \omega_V^*)dM=
\lambda^*_V\int_M u\omega_V^* dM.
\end{eqnarray*}
Thus $\inf_M (-\frac{\Delta_V u}{u})\leq \lambda^*_V$. 
Similarly, we have 
$~\sup_M (-\frac{\Delta_V u}{u})\geq \lambda_V^*.$ \qed
\end{proof}

The above theorem is just   a Barta's type result (for $V=0$ see \cite{Barta}, or \cite{Cha}, III.1, Lemma 1).
\begin{corollary} [Generalized Barta's type inequality] 
\label{Barta12}
For any $u\in \mathcal{D}^+_0$,
\begin{eqnarray}
\label{Barta1}
 \inf_M \left(-\frac{\Delta_V u}{u}\right)\, \leq\,  \lambda_V^*\, \leq \,  
\sup_M \left(-\frac{\Delta_V u}{u}\right)\\ 
\label{Barta2}
 \inf_M \left(-\frac{\Delta^*_V u}{u}\right)\, \leq\,  \lambda_V^*\, \leq \,  
\sup_M \left(-\frac{\Delta^*_V u}{u}\right).
\end{eqnarray}
Equalities hold in (\ref{Barta1}) (in (\ref{Barta2}), respectively) if 
and only if $u=\omega_V$ ($u=\omega^*_{V}$, respectively).
\end{corollary}

The  min-max formula  we will describe next is due to Holland  
\cite{Ho}  on Euclidean domains.
 This formula was reformulated by Godoy, Gossez and Paczka 
\cite{GGP}, using weighted Sobolev spaces.
We give a sketch of the proof, valid at least for the case $\bar{M}$
a coordinate chart, which  formally follows the same steps as in 
\cite{GGP}. We also provide some formulas that we will need.

The weighted Sobolev spaces with weight the square of the distance function to 
$\partial M$,  $ d_{\partial M}(p)=\inf_{x\in \partial M}d(p,x)$, 
for $p\in \bar{M}$,  are defined by
\begin{eqnarray}
 L^2_{\partial}(M) &=&\left\{ u:M\to \mathbb{R} ~\mbox{measurable}:
\int_M d^2_{\partial M}u^2dM<+\infty\right\},~~ \label{sobolev1}\\
 H^1_{\partial}(M) &=& \left\{u\in H^1_{loc}(M): 
\int_M d^2_{\partial M}(u^2+|\nabla u|^2)
dM<+\infty \right\},~~ \label{sobolev2}
\end{eqnarray}
with the weighted Sobolev norms, 
$$\|u\|_{L^2_{\partial}}^2=\int_M d^2_{\partial M}u^2dM,\quad\quad
\|u\|_{H^1_{\partial}}^2= \int_M d^2_{\partial M}(u^2+|\nabla u|^2) 
dM,$$ 
respectively. If $\bar{M}$ is a smooth Euclidean domain, it is shown in 
\cite{GGP} that $H^1_{\partial}(M)$ is continuously embedded into $L^2(M)$ 
and compactly embedded into $L^2_{\partial}(M)$. In  Lemma \ref{omegaVboundary}(3), 
we show this is also true when $\bar{M}$ is a  smooth compact Riemannian domain 
that is  diffeomorphic to an Euclidean domain. This is  clear when 
$M$ is a geodesic ball $B_{r_0}(p_0)$  of $N$ with $r_0<\mathrm{inj}(p)$. 
The exponential map of $N$  defines a  diffeomorphism
$\exp_{p_0}:\bar{D}_{r_0}\to \bar{B}_{r_0}(p_0)$
from the  Euclidean closed $m$-ball $\bar{D}_{r_0}$ of radius $r_0$.
For each $t< r_0$, and  $\xi\in \mathbb{S}^{m-1}$ the distance functions 
to the boundaries are related  by 
$d_{\partial M}(\exp_{p_0}(t\xi))= r_0-t =d_{ \partial D_{r_0}}(t\xi)$.
We also show in the following lemma 
that the principal eigenfunction $\omega_V$ lies in $ \mathcal{D}_{\partial}$. 
Let  $\mathcal{O}$ be  a small tubular neighbourhood of $\partial M$ in $N$  
such that normal minimizing geodesics starting from $\partial M$ are unique.

\begin{lemma}\label{omegaVboundary}
Assume $\partial M$ is a smooth hypersurface of $N$ and $\nu$ is its  
unit outer normal with respect to $\bar{M}$. We have the following:\\[1mm]
 $~(1)$ The distance 
function, $d_{\partial M}:\bar{M}\to [0, +\infty)$, lies in  
$C^{\infty}(\bar{M}\cap \mathcal{O})\cap C_0(\bar{M})$ and satisfies 
$\frac{ d_{\partial M}}{\partial \nu}(p)=-1$, for all $ p\in \partial M$;\\[1mm]
$~(2)$ $\omega_V\in \mathcal{D}_{\partial}\cap \mathcal{D}_0^+\cap H^1_0(M)$; \\[1mm]
$~(3)$ If  $\Phi: \bar{D} \to \bar{M}$ is a diffeormorphism 
from an Euclidean domain $\bar{D}$ onto $\bar{M}$, then we can find
some constants $c_i>0$ such that,
$$ c_1 d_{euc} (x, \partial D)\leq d_{\partial M}(\Phi(x))\leq c_2 
d_{euc} (x, \partial D),$$
 holds for all $\Phi(x)\in \mathcal{O}\cap \bar{M}$. 
Furthermore, $H^1_{\partial}(M)$ is continuously embedded into $L^2(M)$ and 
compactly embedded into $L^2_{\partial}(M)$. Moreover, $\mathcal{D}_{\partial}\subset
H^1_0(M)$. 
\end{lemma}
\begin{proof} (1) and (2).  Locally, $\partial M$ is the hyperplane
of $\mathbb{R}^m$, $\{0\}\times \mathbb{R}^{m-1}$, $\bar{M}$ is the half space 
$\{x_1,x_2,\ldots, x_m\}$  with $x_1\leq 0$, and  
$d_{\partial M}(x_1,\ldots, x_m)=-x_1$.  Hence, locally, $d_{\partial M}$ has smooth 
extensions on a neighbourhood of each point  $p_1\in \partial M$ in $N$. Therefore,
$d_{\partial M}\in C^{\infty}(\bar{M}\cap \mathcal{O})$.
Now, normal geodesics starting from $\partial M$ are of the form 
$\gamma(t)=\exp_{p_1}(\mp t\nu(p_1))$, for $t\in [0,\epsilon)$, 
with $p_1\in \partial M$.  The $-$ sign corresponds to a geodesic lying in
$\bar{M}\cap \mathcal{O}$. We only consider these geodesics.  
Then we have,  $\nabla d_{\partial M}(\gamma(t))=\gamma' (t)$.
This can be shown using Fermi coordinates on $\mathcal{O}$
(\cite{Gray}, Lemma 2.7, Lemma 2.8 (2.25)).
Therefore, $\nabla d_{\partial M}(p_1)=\lim_{t\to 0^+}\nabla 
d_{\partial M}(\gamma(t))=-\nu(p_1)$. It follows that
$\frac{ d_{\partial M}}{\partial \nu}(p_1)=-1$.
Consequentely, $\lim_{p\to p_1}{\omega_V(p)}/{d_{\partial M}(p)}=
-({\partial \omega_V}/{\partial \nu})(p_1)$.  
This equality and the fact that $\omega_V\in K^o$  
implies  $\omega_V\in  \mathcal{D}_{\partial} $.\\[2mm]
(3). Coordinates charts on $\partial D$ are transported by $\Phi$ into 
coordinate charts on $\partial M$. Hence, we may build simultaneously a Farmi 
coordinate system on a tubular neighbourhood $\mathcal{O}_{euc}$ of $\partial D$
and another on $\mathcal{O}$, such that, for each $x\in \partial D$ and $t_1\in [0,\epsilon)$,
 $d(-t_1\nu_{euc}(x),\partial D)=t_1=d_{\partial M}(\exp_{\Phi(x)}(-t_1\nu(\Phi(x)))$,
(\cite{Gray}, Chapter 2, Section 2.1., (2.4)). Now, the first statement follows naturally.  Obviously, the metric $g$ is equivalent to the one induced by the 
Euclidean one via $\Phi$. These two facts imply that the weighted Sobolev norms defined 
for functions $u$ on   $M$ and  for functions $\tilde{u}=u\circ \Phi$  on  $D$ are equivalent. This implies the second last statement of (3) is true for $M$, 
knowing it is true for $D$ (\cite{GGP}, Lemma 4.1). 
Finally, we have $\mathcal{D}_{\partial}\subset
H^1(M)\cap C_0(\bar{M})$. In case $\bar{M}$ is a global chart 
the latter set is contained in $H^1_0(M)$, \cite{SZ}.  \qed \end{proof}
 
For each $u\in  \mathcal{D}_{\partial}$ we consider the continuous functional 
$Q_u:H^1_{\partial}(M)\to \mathbb{R}$, given by
$$ Q_u(v):=\int_M u^2( |\nabla v|^2-g({V},\nabla v)) dM.$$
Now we may present the Holland-Godoy-Gossez-Paczka formula
(\cite{Ho, GGP}).

\begin{theorem}[Min-max integral formula] \label{integralminmax}
If $\bar{M}$ is a regular compact domain of a coordinate chart, then 
\begin{equation}\label{integralminmaxformula}
\lambda_V^*=\inf_{\{u\in \mathcal{D}_{\partial}:~\|u\|_{L^2}=1\}}\La{(} \mathcal{L}(u,u)-
\inf_{v\in H^1_{\partial}(M)}{Q}_u(v)\La{)}.
\end{equation}
Equality is achieved at $u_V=\omega_V\sqrt{G_V}$ (normalized).
Here, $G_V\in H^1_{\partial}(M)$ is the unique solution, up to a multiplicative 
constant, of the integral equation 
\begin{equation}\label{GV}
\int_M
g(\nabla G+ GV, \nabla \phi)\omega_V^2dM=0, \quad \forall \phi\in H^1_{\partial}(M),
\end{equation}
satisfying $0<c_1\leq G_V\leq c_2$, for some positive constants
$c_i$. In particular $u_V\in \mathcal{D}_{\partial}\cap H^1_0(M)$
and
$$\lambda^*_V= \inf_{\{u\in \mathcal{D}_{\partial}\cap H^1_0(M): ~\|u\|_{L^2}=1\}}\La{(}
 \mathcal{L}(u,u) -\inf_{v\in H^1_{\partial}(M)  } Q_u(v)\La{)}. $$
\end{theorem}

The proof is based on two existence results, (A) and (B), 
and makes use of the  \em completing of a square \em algebraic 
inequality (C), described below:  

(A) ~Given $u\in  \mathcal{D}_{\partial}$, there exists $w_u\in H^1_{\partial}(M)$ such that
$Q_u(w_u)=\inf_{v \in H^1_{\partial}(M)}Q_u(v)$. It is unique on $ H^1_{\partial}(M)$ up to an 
additive constant (a.e.). Computing the Euler Lagrange equation we see that, 
\begin{equation}\label{existence-wu}
Q_u(w_u)=-\int_M|\nabla w_u|^2u^2 dM= -\frac{1}{2}\int_M
g(V,\nabla w_u)u^2dM\leq 0.
\end{equation}
In case $V=0$ we must have  $w_u=0$ up to a constant (a.e.), and the min-max 
formula is the usual Rayleigh formula for the first eigenvalue 
of the $0$-Laplacian.

(B)~Given $u\in \mathcal{D}_{\partial}$,  there exists $G_u\in H^1_{\partial}(M)$ such that
\begin{equation}\label{existence-Gu}
B(G_u,\phi):=\int_M g(\nabla G_u + G_u V, \nabla \phi)u^2 dM=0, 
 \quad\forall \phi\in H^1_{\partial}(M).
\end{equation}
It is unique up to a multiplicative positive constant, 
satisfying $0<c_1\leq G\leq c_2$ for some positive constants $c_i$.

(C) For all $X,Z\in T_pM$, 
\begin{equation} \label{square}
-|X|^2-g(V,X) \leq g(X,Z)+ \frac{1}{4}|V+Z|^2.
\end{equation}
Equality holds if and only if $Z=-2X-V$.

The proof of (A) on the existence of a unique minimum $w_u$ relies on the  
strict convexity of the integrand $F(P,v,p)=u^2(|P|^2-g(V(p),P))$ in the 
variable $P$, and on the coerciveness of $Q_u $ on the subspace
$~H^1_{\partial, B}(M)=\{u\in H^1_{\partial}(M):\int_B u=0\}$. Here,  $B$ is some 
fixed  small ball in the interior of $M$.  Coerciveness of $Q_u$ can be 
shown  using the compactness of the embedding $H^1_{\partial}(M)$ into 
$L^2_{\partial}(M)$, 
given in Lemma \ref{omegaVboundary}.  
A minimum $w_u$ is a critical point  of $Q_u$.  It is a (weak) solution 
of the degenerate elliptic second order differential equation 
$\mathrm{div}^0(u^2(2\nabla w-V))=0$,  for $w\in H^1_{\partial}(M)$. 
A critical point $w_u$ satisfies  $Q_u(w_u)=-\int_Mu^2|\nabla w_u|^2\leq 0$. 
The proof of the second existence result (B) is  more complex because one 
has to find a positive solution. A solution is a weak solution of the 
degenerate elliptic second order differential equation 
$\mathrm{div}^0(u^2(\nabla G+VG))=0$, that is, it satisfies
$B(G,\phi)=0$ for all  $\phi\in H^1_{\partial}(M)$. In \cite{Ho},
Holland proved the existence of $G\geq 0$ by finding 
an ergodic measure with probability density $G$. The alternative proof in 
\cite{GGP} consists of taking $\epsilon $ sufficiently large such that 
$B_{\epsilon}(G,\phi)=B(G,\phi) + \epsilon \int_M G\phi u^2dM$ 
is coercive for the $H^1_{\partial}$-norm. Its inverse operator 
$T_{d,\epsilon}:L^2_{\partial}(M)\to L^2_{\partial}(M)$ can be shown to be compact, 
and satisfies the following property: 
if $0\neq f\geq 0$, then $\mathrm{ess} \inf_B T_{d,\epsilon}f >0$, 
for any open domain $B$ with compact closure in the interior of $M$. 
This implies that the compact operator $T_{d,\epsilon}$ is a  positive 
and irreducible operator on $L^2$ spaces, in the sense of Schwartz
\cite{Schw}, which forces the spectral radius of $T_{d,\epsilon}$ to be 
positive. This is a sufficient condition for the existence of a principal 
eigenvalue, and a principal eigenfunction $G_u\geq 0$. Performing a Moser 
type iteration technique from $B(G_u,\phi)=0$, leads to the conclusion that $G_u$ is 
uniformly bounded for all weighted $L^p_{\partial}$-norms. Consequently, 
$G_u$ is bounded from above by a positive constant $c_2$. 
Coerciveness of a related modified operator implies $G_u\geq c_1>0$.  
The divergence of the vector field $U=-u^2\nabla\log \omega_V$ is given by
$$- \mathrm{div}^0(U)
= g\LA{(}\frac{2u}{\omega_V}\nabla u -\frac{u^2}{\omega_V^2}
\nabla \omega_V, \nabla \omega_V\LA{)} 
-u^2\lambda_V^*+ \frac{u^2}{\omega_V}g(V,\nabla \omega_V).$$
Integrability of  $\mathrm{div}^0(U)$  on ${M}$ follows from  
 the properties of $\omega_V/d_{ \partial M}$ and  
$u/d_{\partial M}$ near $\partial M$.
On the other hand, $U$ continuously  extends to zero  on $\partial M$.
Considering for each $\epsilon>0$, 
$M_{\epsilon}=\{p\in M: d_{ \partial M}(p)\geq \epsilon\}$,
 and $\nu_{\epsilon}$ the outward unit of its boundary, we have
$$\int_M \mathrm{div}^0(U)dM =\lim_{\epsilon\to 0}\int_{M_{\epsilon}} 
\mathrm{div}^0(U)dM=\lim_{\epsilon\to 0}
\int_{\partial M_{\epsilon}} g(U,\nu_{\epsilon})dM=0.$$
This fact, and  following \cite{GGP},  taking the $u^2dM$ integration 
of the algebraic inequality (\ref{square}) in (C) with, 
$X=-\nabla \log \omega_V$, 
$Z=-V+2\nabla (\log (u)+ w_u)$,
gives the inequality
$\lambda^*_V\int_M u^2dM \leq \mathcal{L}(u,u) -Q_u(w_u)$.  
Equality holds at
\begin{equation}\label{achieved-uV}u=u_V:=\omega_V\sqrt{G_V},
\end{equation}
where $G_V$ is the solution $G_u$ given in (B) with respect to 
$u=\omega_V$, with 
\begin{equation}\label{achieved-wuV}
w_{u_V}=-(\log G_V)/2.
\end{equation}
This solution  $u_V$ lies in $H^1_0(M)$. To see this we first recall 
that $\omega_V\in \mathcal{D}_{\partial}$. Now, $G_V$ is a weak solution 
of an elliptic operator of second order with smooth coefficients on any  
subdomain $\Omega$ with  smooth compact closure in the interior of $M$. 
Moreover, $G_V\in H^1(\Omega)$ and it is bounded. Hence, 
$G_V\in C^{\infty}(\Omega)$ (cf.\ \cite{Au}, Theorem 3.55). 
In particular, $G_V\in C(M)$. From 
$$|\nabla (\omega_V\sqrt{G_V})|^2\leq 2|\nabla \omega_V|^2|G_V|
+\frac{1}{2}\frac{|\omega_V|^2}{d^2_{\partial M}}\frac{|\nabla G_V|^2}{G^2_V}
d^2_{\partial M},$$
we conclude that $u_V\in H^1(M)$, and so $u_V\in \mathcal{D}_{\partial}\cap C_0(\bar{M})$.
Consequentely, $u_V\in H^1_0(M)$ (cf. \cite{SZ}).\\

In the particular case of the Bakry-\'{E}mery Laplacian, straightforward 
computations prove the following.
\begin{proposition} \label{Proposition5.4} If $~V=\nabla f$, then for any
$u\in \mathcal{D}_{\partial}$, $w_u=f/2$ and $G_u=e^{-f}$. Moreover 
$$\lambda^*_V=\inf_{\tilde{u}\in \mathcal{D}_{\partial}}
\frac{\mathcal{L}(\tilde{u}, \tilde{u})-\inf_v Q_{\tilde{u}}(v)}
{\int_M \tilde{u}^2dM} = \inf_{\tilde{u}\in \mathcal{D}_{\partial}}
\frac{\int_M|\nabla \tilde{u} + \tilde{u}\frac{1}{2}\nabla f|^2dM}
{\int_M \tilde{u}^2dM}.$$
The infimum is achieved at 
$\tilde{u}=\omega_f e^{-f/2} \in \mathcal{D}_{\partial}\cap H^1_0(M)$.
Writing $u=\tilde{u}e^{f/2}$, and recalling that 
$C^{2}_c(M)\subset \mathcal{D}_{\partial}$ is dense in $H^1_0(M)$, we obtain the 
Rayleigh variational characterization for the first eigenvalue of the 
$f$-Laplacian given in (\ref{Rayleigh f}). That is, $\lambda^*_V=\lambda_f$,
and  $u=\omega_f\in \mathcal{D}_{\partial}\cap H^1_0(M)$ is
the $\lambda_f$-eigenfunction for $-\Delta_f$.
\end{proposition}
\begin{corollary} If $V=0$, the min-max formula in Theorem \ref{integralminmax} 
reduces to the Rayleigh variational characterization of the 
first eigenvalue of $-\Delta_0$.
\end{corollary}

\section{Comparison results}

Let $\bar{M}$ be a smooth compact domain endowed with a smooth vector field 
$V$ for which the min-max formulas of Section 5 hold.  This is the case
when $\bar{M}$ is the domain of a smooth coordinate chart. Our first proposition 
is a straightforward application of the min-max integral formula for the 
principal eigenvalues $\lambda_V^*$ and  $\lambda_0^*$,  the second 
one  for  $V=0$. Let $\omega_V$ and $\omega_0$ be the respective 
$L^2$-unit principal eigenfunctions, and $w_{\omega_0}\in H^1_{\partial}(M)$ the 
function that realizes $\inf_v Q_{\omega_0}(v)$ in the  integral inequality 
(\ref{integralminmaxformula}).

\begin{proposition} \label{simple} 
The following inequalities hold:
$$ \lambda^*_V+\frac{1}{2}\int_M(\mathrm{div}^0(V)-2|\nabla w_{\omega_0}|^2)
\omega_0^2~dM ~ \leq ~ \lambda^*_0~ \leq ~\lambda^*_V 
+\frac{1}{2}\int_M \mathrm{div}^0({V})\omega_V^2\, dM. $$
Furthermore, equality holds for the right hand side inequality if and only if 
$\omega_V=\omega_0$. In this case,  
$\lambda_V^*-\lambda_0^*=g(V, \nabla\log \omega_0)$ on $M$, 
$g(V,\nu)=0$ on $\partial M$, and $\int_M \mathrm{div}^0(V)dM=0$, where
$\nu$ is the unit outer normal of $\partial M$.
Equality  holds for the left hand side inequality if and only if
$\omega_V=\alpha \omega_0 e^{w_{\omega_0}}$, where $\alpha$ is a 
normalizing constant.
\end{proposition}
\noindent
\begin{proof}  Using the Rayleigh characterization 
of $\lambda_0^*$ applied to $\omega_V$, and Stokes's theorem,
\begin{eqnarray*}
\lambda_0^* &\leq& \int_M \|\nabla \omega_V\|^2dM\\ 
&=&  \lambda^*_V \int_M \omega_V^2dM 
-\frac{1}{2}\int_Mg({V},\nabla \omega_V^2)dM=
 \lambda^*_V +\frac{1}{2}\int_M \mathrm{div}^0({V}) \omega_V^2\,  dM.
\end{eqnarray*}
Equality holds if and only if $\omega_V=\omega_0$. In this case,
$$\lambda_V^*\omega_V=-\Delta_V\omega_V=-\Delta_V\omega_0=
-\Delta_0\omega_0 +g(V,\nabla\omega_0)=\lambda_0^*\omega_0
+g(V,\nabla\omega_0).$$
Thus,  $(\lambda_V^*-\lambda_0^*)\omega_0=g(V, \nabla \omega_0)$. 
Consequentely, along $\partial M$, $ g(V, \nabla \omega_0)= 
\frac{\partial \omega_0}{\partial \nu}g(V,\nu)$ must vanish.
This is possible only  if $g(V,\nu)=0$. 
Now,  applying the min-max integral formula for $\lambda_V^*$ with respect to
$\omega_0$ we get the left hand side inequality.  Equality holds if and only if 
$\omega_0=\alpha^{-1}u_V= \alpha^{-1}\omega_V\sqrt{G_V}$, as 
seen in (\ref{achieved-uV}), for $\alpha^2=\int_M u_V^2dM$. 
Since $Q_{u_V}=\alpha^2Q_{\omega_0}$, then $w_{u_V}=w_{\omega_0}$.
By (\ref{achieved-wuV}), $\sqrt{G_V}=e^{-w_{u_V}}=e^{-w_{\omega_0}}$. \qed
\end{proof} 

As a consequence of the previous proposition and its proof, and of Proposition 
\ref{Proposition5.4}, we have the following two corollaries: 
 
\begin{corollary}\label{lambdav-lambda0} If $\mathrm{div}^0(V)\leq 0$, then 
$\lambda^*_0\leq \lambda^*_V$. Equality of the eigenvalues  holds if and only if
 $\mathrm{div}^0(V)=0$ and  $\omega_V=\omega_0$. In this case
$V\bot \nabla \omega_0$ on $M$ and 
$V\bot \nu$ along $\partial M$.\end{corollary}
\begin{corollary}\label{lambdav-lambda0-f}
 Let us suppose $V=\nabla f$, for some  $f\in C^{\infty}(\bar{M})$.\\
$(1)$ Assume $\Delta_0f\leq 0$. Then  $\lambda^*_0\leq \lambda_{ f}$, and
equality holds if and only if $f$ is a harmonic function  and 
$\omega_0=\omega_f$.
In this case, $\nabla f\bot \nabla \omega_0$  pointwise on $M$.\\[1mm]
$(2)$ Assume for some constant $\epsilon>  0$, 
$\Delta_0f\geq \frac{\epsilon}{2}|\nabla f|^2$ holds. Then  
$\lambda_0^*\geq \lambda_{\epsilon f}$, and  equality holds if and only 
if  $\Delta_0f= \frac{\epsilon}{2}|\nabla f|^2$.
\end{corollary}

\noindent
Proposition \ref{simple} with $V=\epsilon \nabla f$, and   $\epsilon>0 $ a constant, 
give us 
$$\frac{\epsilon}{2}\int_M(\Delta_0f -\frac{\epsilon}{2}|\nabla f|^2)
\,\omega^2_0\, dM~\leq~ \lambda_0^*-\lambda_{\epsilon f}
~\leq~ \frac{\epsilon}{2}\int_M \Delta_0f \,\omega_{\epsilon f}^2\, dM. $$
Similar reasoning for $\epsilon <0$ leads to  the following consequence for 
the Bakry-\'{E}mery  first eigenvalue.
\begin{corollary} If for some sequence $\epsilon_i\to 0$, 
the limit $\lambda' _f=\lim_{\epsilon_i\to 0}\frac{1}{\epsilon_i}(
\lambda_{\epsilon_if}-\lambda_0^*)$ exists, then 
$-\frac{1}{2}(\sup_M \Delta_0 f) \leq \lambda' _f\leq 
-\frac{1}{2}(\inf_M \Delta_0 f).$
Consequently, if $f\in C^{\infty}(\bar{M})$ is a harmonic function on $M$, 
or more generally, $\Delta_0f=2c_0$ a constant,  then
$\frac{d}{d\epsilon}\la{|}_{\epsilon=0}\lambda_{\epsilon f}$ exists and
it is equal to $-c_0$.
\end{corollary}

Next we will define suitable  models spaces, based on 
pointwise estimates of the radial curvatures  and  of the radial component of 
$V$. These model spaces will establish estimates for the principal eigenvalue 
of a geodesic ball of $N$ by comparing it with the corresponding ones of the
model spaces.

The exponential map of $N$ from a given point $p_0$ is a smooth diffeomorphism
$\exp_{p_0}:\mathcal{D}_{p_0}\to N\backslash C(p_0)$ from the star-shaped open set 
$ \mathcal{D}_{p_0}=\{ t\xi:\,  0\leq t\leq d_{\xi}, \, \xi \in S^{m-1}_{p_0}
\subset T_{p_0}N \}$, 
onto the open dense set $ N\backslash C(p_0)$ of $N$, where $C(p_0)$ is the cut locus 
at $p_0$ and $d_{\xi}$ is the largest $t$ for which $\gamma_{\xi}(s)
=\exp_{p_0}(s\xi)$ is a minimizing geodesic for all $0< s\leq t$.
This diffeomorphism defines on $ N\backslash C(p_0)$ the geodesic coordinate chart,
 $\hat{\Theta}(t,\xi)=\exp_{p_0}(t\xi)$.
In these coordinates the metric $g$ can be expressed as
$$ g(\hat{\Theta}(t,\xi))=dt^2+|\mathcal{A}(t,\xi)d\xi|^2,
\quad\quad \forall t\xi\in  \mathcal{D}_{p_0}.$$
Here,  $\mathcal{A}(t,\xi):\xi^{\bot}\to \xi^{\bot}$  is the linear operator
given by $\mathcal{A}(t,\xi)\eta=\tau_t^{-1}Y_{\eta}$, where
$Y_{\eta}(t)=d(\exp_{p_0})_{(t\xi)}(t\eta)$ is the Jacobi field along
the geodesic $\gamma_{\xi}(t)$, with initial conditions, $Y_{\eta}(0)=0$,
$\nabla_{\partial_t}Y_{\eta}(0)=\eta$, and
$\tau_t:T_{p_0}M\to T_{\gamma_{\xi}(t)}$ is the parallel transport
along $\gamma_{\xi}$. It satisfies the Jacobi equation
$\mathcal{A}''+\mathcal{RA}=0$, with $\mathcal{A}(0,\xi)=0$,
$\mathcal{A}' (0,\xi)=Id$, where $\mathcal{R}(t):\xi^{\bot}\to \xi^{\bot}$
is the self-adjoint operator, $\mathcal{R}(t)\eta=(\tau_t)^{-1}
R(\gamma_{\xi}'(t), \tau_t\eta)\gamma_{\xi}'(t)$. The trace of $\mathcal{R}(t)$
is  the radial Ricci tensor, 
$\mathrm{Ricci}_{(\gamma_{\xi}(t))}(\gamma_{\xi}'(t),\gamma_{\xi}'(t))$. 
We define a non-negative smooth function $J$  on $\mathcal{D}_{p_0}$, such that
 $$J^{m-1}= \mathrm{det} \mathcal{A}.$$
Let $dS$ be  the volume element of  $\mathbb{S}^{m-1}$, and $r(p)=d(p,p_0)$
the intrinsic distance of $p$ to $p_0$ in $N$. The square $r^2(p)$ 
 is smooth on $N\backslash C(p_0)$ (cf.\ \cite{Gray}, 
Section 3.2), and the gradient $\nabla r$ is a unit vector field. For $t>0$, 
it satisfies the equality, $\nabla r(p)=\gamma'_{\xi}(t)$, for $p=\gamma_{\xi}(t)$,
and defines the  radial direction $\partial_t(p)=\nabla r(p)$ at each $p\neq p_0$. 
In these geodesic coordinates $(t,\xi)$, $~dV_M=J^{m-1}(t,\xi)\,dt\,dS$ expresses the 
volume element of $M$. The function $J$ satisfies the following 
equations and inequalities (cf.\cite{FMS})
\[
\Delta_0 r=\partial_t\ln(J^{m-1}),\quad\quad
\partial_t\Delta_0 r +\|\mathrm{Hess}\, r\|^2
=-\mathrm{Ricci}(\partial_t,\partial_t),
\]
\[
\left\{ \begin{array}{l} 
 (m-1)J''(t,\xi) + \mathrm{Ricci}(\gamma_{\xi}'(t), \gamma_{\xi}'(t))J(t,\xi)\, 
\leq\,  0,\\
 J(0,\xi)=0, \quad J' (0,\xi)=1.
\end{array}\right.
\]
Note that $r\Delta_0r=\frac{1}{2}(\Delta_0 r^2-1)$ is  smooth, 
and applying L'H\^{o}pital's rule, 
\begin{equation}\label{distanceeqs}
 \lim_{t\to 0^+} ~r\,\Delta_0 r(t,\xi)~=~
\lim_{t\to 0^+} ~(m-1)\,\frac{t}{J(t,\xi)}\, J'(t,\xi)~=~(m-1).
\end{equation}
We are assuming $r_0<\mathrm{inj}(p_0)$, so that 
$\bar{M}\subset N\backslash C(p_0)$, and $d_{\xi}\geq
\mathrm{inj}(p_0)$, $\forall \xi$. The restriction of the geodesic
coordinates, $\hat{\Theta}:[0,r_0]\times \mathbb{S}^{m-1}\to \bar{M}$, 
defines the spherical geodesic coordinate of $M$ centered at $p_0$.
It satisfies $\frac{d\hat{\Theta}}{dt}=\nabla r$. We have the following identities
holding for any function $\phi\in C^1(M)$, 
$$\frac{d\phi}{dt}=\phi'(t,\xi)=\frac{d(\phi \circ \hat{\Theta})}{dt}=
g(\nabla \phi, \nabla r)=\partial_t\phi.$$
For any radial function $F$ on $M$, i.e  $F(p)=T(r(p))$, where 
$T:[0,r_0]\to \mathbb{R}$ is of class $C^{2}$,  satisfying $T' (0)=0$, 
 we have (cf. \cite{Cha})
\begin{equation}\label{radialeqs}
\begin{array}{l}
\nabla F(p)=T'(r(p))\nabla r= T'\partial_t\\
\Delta_0F(p)=T''(r(p))+(\partial_r\log J^{m-1}) T' (r(p))= T'' 
+ \Delta_0 r\, T'. 
\end{array}
\end{equation}

We decompose the vector field  $V$ as ${V}=V_{rad}+V_s$, where
$V_{rad}=h_1(t,\xi)\nabla r$, with $h_1(t,\xi):=g({V},\nabla r)$ the radial 
component of ${V}$, and $V_s$ the $g$-orthogonal complement of $V_{rad}$.
We say $V$ is a radial vector field if $V=V_{rad}$. It is smooth if  
$h(0,\xi)=0$  for any $\xi$, to be more precise, if $h(t,\xi)=
t\tilde{h}(t,\xi)$ for some smooth function $\tilde{h}$. 
If $V$ is not radial, $V_{rad}$ is not assumed to be smooth. 

As in \cite{FMS}, we consider a model space  
$N^{\rho}=[0,l)\times_{\rho} \mathbb{S}^{m-1}$,
and a geodesic ball $M^{\rho}$ centered at the origin $p_{\rho}$ with  
 radius $r_0$. We are assuming $r_0<\min\{l, \mathrm{inj}(p_{0})\}$, 
and  take  a radial vector field $V^{\rho}=h(t)\partial_t=\nabla H$, 
where $H'(t)=h(t)$ and $h(0)=0$. 
On the model space, $\mathcal{A}(t,\xi)=\rho(t)Id$, $J=\rho$,
and $d((t,\xi), \partial M^{\rho})=r_0-t$, for $t\leq r_0$.
The properties of the positive  principal eigenfunction   
$\omega_{\rho,H}(r)$ on  $M^{\rho}$  are described in  
Proposition \ref{radialmodel}.

The ratio of the volume elements of $M$ and $M^{\rho}$ 
is a fundamental tool to derive comparison 
results based on relations between radial curvatures:
\begin{equation} \label{ratio}
\theta(t,\xi)~=~\frac{dM(p)}{dM^{\rho}(p)}~=~
\left[\frac{J(t,\xi)}{\rho(t)}\right]^{m-1},\quad \theta(0,\xi)=1,
\end{equation}
where $p=\hat{\Theta}(t, \xi)$. Comparison on radial curvatures corresponds  to  
 nondecreasing or nonincreasing $\theta(t,\xi)$ on $[0,r_0)$, and consequent 
 inequality on the volumes of the geodesic balls of radius $t<r_0$  
(see \cite{FMS}, generalized Bishop's comparison Theorems 4.2 and 3.3). 
We start by recalling the  comparison result of  \cite{FMS} for the first 
eigenvalue of $\Delta_0$ on a geodesic ball with radial sectional curvatures  
bounded from above by those of the  model space.
\begin{theorem}  Assume the radial sectional curvatures of  
$\bar{M}=\bar{B}_{r_0}(p_0)$ are bounded from above by the ones of the model 
space $(\bar{M}^{\rho}, H=0)$, that is,
$K(\partial_t,X)\leq -\frac{\rho''(t)}{\rho(t)}$, for all unit $X\in T_pM$
orthogonal to $\partial_t(p)$. Then
$\lambda_0^*\geq \lambda_{\rho, H=0}$, and  equality holds
if and only if $M$ is isometric to $M^{\rho}$.
\end{theorem}
\noindent
This is an extension of Cheng's comparison result reduced to the
case $-\frac{\rho''(t)}{\rho(t)}=constant$,  the case of a space 
form \cite{Cheng}.
We now extend this  result to the $V$-Laplacian, obtaining Theorem
\ref{Theorem 1.1}. 
\begin{theorem} On $\bar{M}=\bar{B}_{r_0}(p_0)$ it is given a vector field  
$V$,  where $V_{rad}=h_1(t,\xi)\partial_t$.  Assume the 
radial sectional  curvatures are bounded from above by the one of 
the model space $(\bar{M}^{\rho},V^{\rho} )$, that is, 
$K(\partial_t,X)\leq -\frac{\rho''(t)}{\rho(t)}$. Additionally, assume 
$h_1(t,\xi)\leq h(t)$. Then $\lambda_V^*\geq \lambda_{\rho,H}$.
If equality holds on the eigenvalues, 
 then  $M$ is isometric to $M^{\rho}$ and $h_1(t,\xi)=h(t)$. 
In that case  $\omega_V=\omega_{\rho, H}$. 
\end{theorem}
\noindent
\begin{proof} By the curvature conditions and the generalized Rauch-Bishop's
comparison theorem (\cite{FMS}, Theorem 4.2) $\theta'(t,\xi)\geq 0$, 
where $\theta$ is defined in (\ref{ratio}). Equivalently,
$$\Delta_0 r = (m-1)(\log J)' \geq \Delta^{\rho}_0r=(m-1)(\log \rho)', $$
with equality if and only if 
$\mathcal{A}=\rho(t)Id$, that is $M$ is isometric to $M^{\rho}$. 
On the spherical geodesic coordinates of $M$, we define
$\tilde{\omega}(t,\xi):=\omega_{\rho,H}(t)$ 
extending the principal eigenfunction of the model space for the 
$\Delta_H^{\rho}$-Laplacian to a radial function on $M$. 
Recall that $\omega_{\rho,H}' (t)<0$ on $(0,r_0]$ and 
$\omega_{\rho,H}(r_0)=\omega'_{\rho, H}(0)=0$. 
It is clear that $\tilde{\omega}\in \mathcal{D}^+_0$ on $M$. 
Then by (\ref{radialeqs}) and Proposition \ref{radialmodel},
\begin{eqnarray}
-\frac{\Delta_V \tilde{\omega}}{\tilde{\omega}} 
&=& -\frac{\Delta_0 \tilde{\omega} 
-g(V,\nabla \tilde{\omega})}{\tilde{\omega}} \nonumber \\
&=& -\frac{1}{\omega_{\rho, H}(t)}\left( \omega''_{\rho,H}(t)
+\La{(}(\log (J^{m-1})(t))' -h_1(t,\xi)\La{)}\omega'_{\rho,H}(t)\right)
\quad \label{equality1}\\
&\geq & -\frac{1}{\omega_{\rho,H}(t)}\left( \omega''_{\rho,H}(t)
+\La{(} \sm{(m-1)}\frac{\rho' (t)}{\rho(t)}
 -h(t)\La{)}\omega'_{\rho,H}(t)\right)
\label{equality2} \\
&=& \lambda_{\rho, H}.\nonumber
\end{eqnarray}
From the generalized Barta's inequality in Corollary \ref{Barta12} 
for the $V$-Laplacian,
$$\lambda^*_V\geq \inf_M -\frac{\Delta_V \tilde{\omega}}{\tilde{\omega}}
\geq \lambda_{\rho,H}.$$ Equality holds if and only if $M$ is isometric to 
$M^{\rho}$, $h_1=h$ and $\omega_V=\omega_{\rho, H}$.\qed
\end{proof}
The radial sectional curvatures do not depend on the vector field $V$. 
The following corollary is an immediate consequence of the proof of the 
above theorem. 
\begin{corollary}  On a geodesic ball $M^{\rho}$ of a model space $N^{\rho}$, 
if the radial part of a vector field $V$ satisfies $h_1(t,\xi)=h(t)=H'(t)$
for a smooth function  $h\in C^{\infty}([0, r_0])$, with $h(0)=0$, then  
$\lambda^*_{V}=\lambda_{\rho, H}$ and $\omega_V(t,\xi)=\omega_{\rho, H}(t)$.
\end{corollary}
\noindent
\begin{proof}  As in the previous proof, we get equality of (\ref{equality1}) 
with (\ref{equality2}), where the last is constant. Thus, we have 
$-\Delta_V\tilde{\omega}= \lambda_{\rho,H}\tilde{\omega}$, and 
$\tilde{\omega}>0$ on $M=M^{\rho}$, $\tilde{\omega}=0$ on $\partial M$, that is,
$\tilde{\omega}$ is a principal eigenvalue on $M$ for the $V$-Laplacian.
Consequently, $\omega_V(t,\xi)=\tilde{\omega}(t,\xi)=\omega_{\rho, H}(t)$
and $\lambda_V^*=\lambda_{\rho, H}$. \qed
\end{proof} 

Now we recall the comparison theorem in \cite{FMS} for $\Delta_0$ with
radial Ricci curvature bounded from below.
\begin{theorem}  Assume the radial Ricci curvature of a geodesic ball 
$M$ of radius $r_0$  is bounded from below by the one
of the model space $M^{\rho}$, i.e 
$\mathrm{Ricci}(\partial_t, \partial_t)\geq -(m-1)\frac{\rho''(t)}{\rho(t)}$. 
Then $\lambda^*_V\leq \lambda_{\rho,H=0}$, and equality holds if and only if
$M$ is isometric to $M^{\rho}$.
\end{theorem}

Next we
extend the above results to the $V$-Laplacian when $V$ is a radial vector field,
 but not necessarily a gradient one. 
That is, $V=h_1(t,\xi)\partial_t$, where $h_1(t,\xi)$ may depend on $\xi$. 
\begin{theorem}\label{Theorem 6.8} 
Let $h(t)=H'(t)$, where $h\in C^{\infty}([0, r_0])$, ${h}(t)\geq 0$ and 
$h(0)=0$. Assume $~V(t,\xi)=h_1(t,\xi)\partial_t$ for a function  
${h}_1\in C^{\infty}(\bar{M})$, satisfying  ${h}_1(0,\xi)= 0$, 
$\forall \xi$, and the following inequalities take place at each $(t,\xi)$,
\begin{eqnarray}\label{LowerBound}
\mathrm{Ricci}(\partial_t, \partial_t) 
&\geq & -(m-1)\frac{\rho''}{\rho},
\nonumber\\
h_1' -\frac{h_1^2}{2}+ h_1\,\Delta_0 r &\geq & 
h'  -\frac{h^2}{2}+ h\,\Delta_0^{\rho} r.
\end{eqnarray}
Then $\lambda_V^*\leq \lambda_{\rho, H}$, and 
equality  holds if and only if $M$ is isometric 
to $M^{\rho}$ and equality holds in  (\ref{LowerBound}). In the latter case, 
$\omega_V(t,\xi)=\omega_{\rho,H}(t)e^{(-H(t)+H_1(t,\xi))/{2}}$, where
$H_1'(t,\xi)=h_1(t,\xi)$. Additionally, if $\rho(t), h(t)$ and $ h_1(t,\xi)$ 
are analytic functions on $t\in [0,r_0]$, then  $h_1(t,\xi)=h(t)$ and 
$\omega_V=\omega_{\rho,H}$.
\end{theorem}
\noindent
The second inequality of (\ref{LowerBound}) is just (\ref{extra}) in 
Theorem \ref{Theorem 1.2}. We need  the following lemmas to prove the 
theorem:
\begin{lemma} \label{lemma 6.8}
If $~V(t,\xi)=h_1(t,\xi)\partial_t$, then 
$\mathrm{div}^0(V)= h'_1 (t,\xi)+ h_1(t,\xi)\Delta^0 r$. 
In this case, for any $u\in H^1_{\partial}(M)$ the solution $w_u$ is described as follows:
\begin{eqnarray}
w_u(t,\xi) &=& \frac{1}{2}\int_0^t h_1(\tau,\xi)d\tau, 
=:\frac{1}{2}H_1(t,\xi),\\ 
 Q_u(w_u)&=& \inf_vQ_{u}(v)= -\frac{1}{4}\int_{M} u^2h^2_1dM.\label{Exact} 
\end{eqnarray}
\end{lemma}
\noindent
\begin{proof}  
We prove the two equations. For any $v\in H^1_{\partial}(M)$,
\begin{eqnarray*}
Q_{u}(v) =\int_Mu^2\La{(}|\nabla v|^2-h_1(t,\xi)\frac{d v}{dt}\La{)}dM\geq 
\int_Mu^2\La{(}|\frac{d v}{dt}|^2-h_1(t,\xi)\frac{d v}{dt}\La{)}dM.
\end{eqnarray*}
Hence,
\begin{equation} \label{infinf}
\inf_vQ_{u}(v)\geq  \inf_v \int_Mu^2\La{(}|\frac{d v}{dt}|^2
-h_1(t,\xi)\frac{d v}{dt}\La{)}dM. 
\end{equation}
The  infimum on the right hand side of (\ref{infinf}) is achieved when
$\frac{dv}{dt}(t,\xi)=h_1(t,\xi)/2$, that is,
$v(t,\xi)=\frac{1}{2}\int_0^th_1(\tau,\xi)d\tau + c(\xi)$, 
 giving an   equality in  (\ref{infinf}). 
We may choose $c(\xi)=0$, defined a.e. on $M$. 
In this case it must be $w_u$. \qed \end{proof}

\begin{remark} {\rm In the above Lemma, if $h_1(t,\xi)\geq 0$, 
then $\mathrm{div}^0(V)\leq 0$ is only possible if $h_1\equiv 0$. Indeed, 
if we assume $\mathrm{div}^0(V)(t,\xi)\leq 0$ is possible for a fixed $\xi$ 
and  all $t\in [0,t_2]$, then $(\log h_1)' \leq (\log(J^{1-m}))'$ on that 
interval. Integration of this inequality on $t_1<t_2$,  gives
$h_1(t_2,\xi)/h_1(t_1, \xi)\leq J^{m-1}(t_1,\xi)/
J^{m-1}(t_2,\xi)$, that is $(h_1(t,\xi)J^{m-1}(t,\xi))' \leq 0$
on $[0, t_2]$. Since $h_1(t,\xi)J^{m-1}(t,\xi)$ vanishes
 at $t=0$,  $h_1(t,\xi)J^{m-1}(t,\xi)
\leq 0$. By assumption $h_1\geq 0$, hence we must have $h_1\equiv 0$.
Thus, we have the following conclusion, where (b) is proved using
a similar reasoning.}
\end{remark}

\begin{lemma} Let $V=h_1(t,\xi)\partial_t$ be a radial  vector field
with $h_1(t_1,\xi)=0$. The following statements hold:\\[1mm]
$(a)$ If $h_1(t,\xi)\geq 0$ on $[t_1, t_2]$, and $\mathrm{div}^0(V)\leq 0$  
then $h_1\equiv 0$ on $[t_1, t_2]$.\\[1mm]
$(b)$ If $h_1(t,\xi)\geq 0$ on $[t_1, t_2]$, then $\mathrm{div}^0(V)\geq 0$ 
if and only if $h_1J^{m-1}$ is a nondecreasing on $[t_1, t_2]$.
\end{lemma}

\begin{lemma} \label{lemma 6.9} Given functions, $u\in C^1_0(\bar{M})$,  and
 $\phi\in C^1(\bar{M})$  satisfying $\phi(p_0)=0$, where $\bar{M}=\bar{B}_{r_0}(p_0)$, 
 we have
$$\int_M \phi\frac{du}{dt}\, dM=-\int_M u(\frac{d\phi}{dt} +\phi\Delta_0r)\,dM.$$
\end{lemma}
\noindent
\begin{proof} 
The vector field $W=u\phi\nabla r$ is continuous on  $\bar{M}$, of class 
$C^1$ on $M\backslash\{p_0 \}$,  and vanishes at $\partial M$. 
 On the other hand, 
$$ \mathrm{div}^0(W)=g(\nabla(u\phi), \nabla r)+ 
u\phi\, \mathrm{div}^0(\nabla r)= \phi\frac{du}{dt} + u\frac{d\phi}{dt}
 + u\phi\, \Delta^0r.$$
By L'H\^{o}pital's rule, for each $\xi$, and using (\ref{distanceeqs})
$$\lim_{t\to 0^+} \phi \Delta^0r(t,\xi)=\lim_{t\to 0^+} 
\frac{\phi(t,\xi)}{t}r(t,\xi) \Delta^0r(t,\xi)=(m-1)\frac{d\phi}{dt}(0,\xi).$$ 
Thus, $\mathrm{div}^0(W)\in L^1(M)$. Applying Stokes's theorem on
$M_{\epsilon}:=M\backslash B_{\epsilon}(p_0)$, for all $\epsilon >0$ small, 
and using the fact that $\partial M_{\epsilon}=
\partial M \cup \partial B_{\epsilon}(p_0)$ we get,
$$\int_M \!\mathrm{div}^0(W)dM=\lim_{\epsilon \to 0}\int_{M_{\epsilon}}
 \!\!\!\mathrm{div}^0(W)dM=\int_{\partial M} \!\!\!g(W,\nu)dS
-\lim_{\epsilon \to 0}\int_{\partial B_{\epsilon}
(p_0)}\!\!\!\!\!\!\!\!\!g(W,\nu^{\epsilon})dS. $$
 Here, $\nu^{\epsilon}$ is the outward unit of $\partial B_{\epsilon}(p_0)$
 and $dS$ is the volume element of hypersurfaces. 
The area $|\partial B_{\epsilon}(p_0)|$ converges to zero when 
$\epsilon \to 0$. Hence,
$$\left|\int_{\partial B_{\epsilon}(p_0)}\!\!\!\!\!\!\!g(W,\nu^{\epsilon})dS 
\right|~\leq~ |\partial B_{\epsilon}(p_0)|\sup_M|W|\to 0,  
\quad \mbox{when}~\epsilon\to 0.$$ 
Since $W$ vanish on $\partial M$,  $\int_M \mathrm{div}^0(W)dM=0$,
which  proves the lemma. \qed
 \end{proof}
 
\noindent
%\begin{pot}
{\sl Proof of Theorem \ref{Theorem 6.8}}
 Let $\omega_{\rho,H}$ be a 
positive principal eigenvector for the $H$-drift Laplacian on the geodesic 
ball $M^{\rho}$ of radius $r_0$ and center at the origin. We take 
$\tilde{w}(t)=\omega_{\rho,H}(t)e^{-{H(t)}/{2}}$. Using the spherical 
geodesic coordinates, we extend it as a function on $M$,
$\tilde{w}(t,\xi):=\tilde{w}(t)$, and normalize the extension to have
$\int_M \tilde{w}^2 dM=1$.  Since $d_{\partial M}(p)= r_0-r(p)$, 
for $p=\hat{\Theta}(t,\xi)$ with $r(p)=t$, and $\omega_{\rho,H}(t)/(r_0-t)$ is 
bounded from below and above by positive constants, we see that 
$\tilde{w}\in \mathcal{D}_{\partial}$ on $\bar{M}$. We apply the min-max formula 
(\ref{integralminmaxformula}) for the principal eigenvalue of $M$  
to the radial function  $\tilde{w}.$  Thus,
\begin{eqnarray*}
\lambda_V^* +\inf_vQ_{\tilde{w}}(v) &\leq& \int_M\la{(}
\|\nabla \tilde{w}\|^2+g(V,\nabla \tilde{w})\tilde{w}\la{)}dM. 
\end{eqnarray*}
We express the integrals using the spherical geodesic coordinate
$\hat{\Theta}$, giving
\begin{eqnarray*}
\int_M g({V},\nabla \tilde{w})\tilde{w} dM &=& 
\int\limits_{\xi\in{S}^{m-1}}\La{[}\int\limits^{r_0}_{0}
g(V,\partial_t)\frac{d\tilde{w}}{dt}\tilde{w}(t)
{\rho(t)^{m-1}}\theta(t,\xi)dt\La{]}dS,
\end{eqnarray*}
and
\begin{eqnarray}
\lefteqn{ \int_M\|\nabla \tilde{w}\|^2 dM
 = \int_{\mathbb{S}^{m-1}}\La{[}\int_0^{r_0}(\frac{d\tilde{w}}{dt})^2
\rho^{m-1}(t)\theta(t,\xi)dt\,\La{]}\, dS =} \nonumber\\
 &=&\!\!\!\!\! \int_{\mathbb{S}^{m-1}}\!\!
\LA{\{} w\frac{d\tilde{w}}{dt}\rho^{m-1}(t) \theta(t,\xi)\LA{]}^{r_0}_0
 -\int_0^{r_0}\!\frac{\tilde{w}}{\rho^{m-1}\theta} \frac{d}{dt}\La{[}
\rho^{m-1}\theta \frac{d\tilde{w}}{dt}\La{]} \rho^{m-1}\theta dt 
\LA{\}} dS \nonumber\\
 &=& \int_{\mathbb{S}^{m-1}}\LA{\{}  \tilde{w} \frac{d\tilde{w}}{dt}\rho^{m-1}
\theta(t,\xi)\LA{]}^{r_0}_0 \label{negative1} \\
&&\quad-\int_0^{r_0}\tilde{w}\cdot\La{(}\frac{d^2\tilde{w}}{dt^2}+
(m-1)\LA{[}\frac{\rho'}{\rho} 
+ \frac{\rho}{J}\La{(}\frac{J}{\rho}\La{)}'\LA{]}
\frac{d\tilde{w}}{dt}\La{)}\rho^{m-1}\theta dt\LA{\}}dS. \nonumber
\end{eqnarray}
From equation (\ref{principaleigenmodel}) we have,
\begin{equation}\label{ODEtildew}
\tilde{w}'' +(m-1)\frac{\rho'(t)}{\rho(t)}\tilde{w}' 
+ (B(t)+\lambda_{\rho,H}) \tilde{w}(t)=0,
\end{equation}
where
$$B(t):= \frac{1}{2}h'(t) - \frac{1}{4}(h(t))^2 +\frac{(m-1)}{2}
\frac{\rho'(t)}{\rho(t)}h(t).$$
Furthermore, Proposition \ref{radialmodel} and the assumption $h\geq 0$ imply 
that for $t>0$,
$$\tilde{w}' (t)=e^{-\frac{H(t)}{2}}(\omega'_{\rho,H}(t)
-\frac{h(t)}{2}\omega_{\rho,H}(t))< 0, $$
and $\tilde{w}'(0)=\tilde{w}(r_0)=0$.  Hence, the term (\ref{negative1}) 
vanishes. Under  the curvature conditions, we apply the generalized Bishop's 
comparison theorem I (see \cite{FMS}, Theorem 3.3) to get
\begin{equation}\label{BishopIa}
(J/\rho)' (t,\xi)\leq 0,
\end{equation}
 with equality if and only if $\mathcal{A}=\rho(t) Id$ and
$M$ isometric to $M^{\rho}$. Consequently,
\begin{equation}\label{BishopIb}
\Delta_0r =(m-1)(\log J)' \leq \Delta_0^{\rho}r =(m-1)(\log \rho)'.
\end{equation}
From (\ref{ODEtildew}) and
(\ref{BishopIa}),  and using the fact that $\tilde{w}' \leq 0$, we arrive at
\begin{eqnarray}
\lefteqn{ \lambda_V^* +\inf_vQ_w(v) \leq}\nonumber\\
&\leq&-\int_{\mathbb{S}^{m-1}}\LA{[}\int_0^{r_0}\tilde{w}\LA{\{}
\frac{d^2\tilde{w}}{d^2t}+\La{(}\frac{(m-1)\rho'}{\rho}
-g(V,\partial_t)\La{)}
\frac{d\tilde{w}}{dt}\La{\}}J^{m-1}dt \LA{]}dS\nonumber \\
&=& -\int_M\tilde{w}\La{\{}
-\lambda_{\rho,H}\tilde{w}-B\tilde{w}-
h_1\frac{d\tilde{w}}{dt}\La{\}}dM. \label{applylemma}
\end{eqnarray}
In the last integral the function $B(t)$ is considered extended as 
a radial function on $M$  via  $\hat{\Theta}$. Applying 
Lemma \ref{lemma 6.8}, we have 
$\inf_{v}Q_{\tilde{w}}(v)=-\frac{1}{4}\int_M \tilde{w}^2h_1^2dM$.
Applying Lemma \ref{lemma 6.9}  on the last term of (\ref{applylemma}), 
with $\phi=\frac{h_1}{2}$ and $u=\tilde{w}^2$, we obtain the inequality
\begin{eqnarray*}
 \lambda_V^* 
&\leq & \lambda_{\rho,H}+
\int_{M}\tilde{w}^2\LA{(}\frac{h^2_1}{4} +
\frac{h'}{2} -\frac{h^2}{4}+\Delta_0^{\rho}r\,\frac{h}{2}-\frac{h'_1}{2}
-\Delta_0r\, \frac{h_1}{2}  \LA{)}dM.
\end{eqnarray*}
The assumption (\ref{LowerBound}) imples $\lambda_V^*\leq\lambda_{\rho,H}$.
Equality holds if and only if  $J/\rho \equiv 1$, $M$ is isometric 
to $M^{\rho}$, equality holds in (\ref{LowerBound}), and the min-max formula 
is achieved at $\tilde{w}$. In this case $u_V=\omega_V\sqrt{G_V}=\tilde{w}
=\omega_{\rho,H}(t)e^{-\frac{H(t)}{2}}=u_H$, where $u_H=u_{\nabla H}$.
As we have seen in (\ref{achieved-wuV}), $w_{u_V}= -\frac{1}{2}\log G_V$, 
and  by Lemma \ref{lemma 6.8}, $w_{u_V}=\frac{1}{2}H_1$. 
Thus $\sqrt{G_V}=e^{-\frac{H_1}{2}}$, which proves the relation between 
$\omega_V$ and $\omega_{\rho,H}$ stated in the theorem.

Now we assume $\bar{M}$ is isometric to  $\bar{M}^{\rho}$, and for each $\xi$,
 $h_1(t,\xi)$, $h(t)$, and  $\rho(t)$ 
are analytic functions on $[0, r_0]$. Next we  show that under these conditions,  
and the initial condition $h_1(0,\xi)=h(0)=0$, 
equality in  (\ref{LowerBound}) implies $h_1(t,\xi)=h(t)$, $\forall t,\xi$, and 
consequently $\omega_V=\omega_{\rho,H}$.
For each fixed $ \xi\in \mathbb{S}^{n-1}$, we define a Riccati equation 
on $h_1(t,\xi)$,
\begin{equation}\label{Riccati}
h'_1 = q_0 + q_1h_1+q_2h_1^2,
\end{equation}
with initial condition $ h_1(0, \xi)=0$, where the coefficients are given by
\begin{equation}\label{Riccaticoeff}
q_0 = h' -\frac{h^2}{2}+ h\Delta^{\rho}_0r, \quad
q_1 = -\Delta^{\rho}_0r, \quad
q_2 =\frac{1}{2}.
 \end{equation}
From $\Delta^{\rho}_0r=(m-1)\frac{d}{dt}\log{\rho}$, we conclude that
$q_1$ has a simple pole at $t=0$. Since $h(0)=0$, the 
coefficient $q_0$ is analytic  at $t=0$ with $q_0(0)=mh' (0)$.
Equality in (\ref{LowerBound}) is equivalent to $h_1$ satisfing 
(\ref{Riccati}) with (\ref{Riccaticoeff}). The function $h$ trivially 
solves (\ref{Riccati}) with the same initial condition $h(0)=0$. We use 
the Frobenius method to show  there is uniqueness of solutions of  
equation (\ref{Riccati}) with the same initial condition at the regular 
singular  point $t=0$.  We set  
\begin{equation} \label{Riccatichange}
h_1=-\frac{1}{q_2}\frac{u' }{u}=-2\frac{u' }{u},
\end{equation} 
assuming $u(0)=1$ without loss of generality, and  equation (\ref{Riccati}) 
turns into
\begin{equation}\label{Riccati2ndorder}
\left\{\begin{array}{l}
u''-q_1u' +\frac{q_0}{2}u =0, \\[1mm]
u(0)=1, \quad u'(0)=0.\end{array}\right. \end{equation}
The indicial equation is given by 
$ I(\alpha):=\alpha^2 +(P(0)-1)\alpha + Q(0)=0$,
where $P(0)=\lim_{t\to 0}-tq_1(t)$ $=m-1$ and $Q(0)=\lim_{t\to 0}t^2q_0=0$.
Hence  $I(\alpha)=\alpha(\alpha +m-2)=0$.  The roots are $\alpha_1=0$, 
and $\alpha_2=2-m$. If $m=2$ we have a double root $\alpha_i=0$, 
and so a unique analytic solution $u(t)$  exists, uniquely determined 
by its value at $t=0$ (cf.\ the detailed exposition in  \cite{Ghorai}). 
Now, $H_1(t,\xi)=\log (u^{-2}(t))+ C'$. We fix $u_0$ the solution 
corresponding to $h$, with $u_0(0)=1$, and  $u_0' (0)=0$.
Any other  solution $u_1(t)$  must be $u_0$. This corresponds to 
$h_1(t,\xi)=h(t)$. If $m\geq 3$, then  $\alpha_1-\alpha_2=m-2$.
One of the solutions is given by $u_1$ as in case $m=2$. The other 
type of solution is  of the form
$u_2(t)= c \log t\,\,  u_1(t) + t^{2-m}\sigma(t)$, with $c$ a constant 
and  $\sigma(t)$ an analytic function satisfying $\sigma(0)\neq 0$, 
giving a   corresponding solution $h_2(t,\xi)$  unbounded  at $t=0$, 
thus it  cannot satisfy the initial condition. This  completes the  
uniqueness proof. \qed 
%\end{pot}

\begin{remark}  {\rm The radial function $\tilde{w}$ in  the previous proof
is in  $C^{\infty}(\bar{M})$ and vanishes on the boundary $\partial M$. 
We could apply Lemma \ref{lemma 6.9} to  $\phi= \frac{d\tilde{w}}{dt}$ 
and $u=\tilde{w}$  to get 
$$\int_M\|\nabla \tilde{w}\|^2 dM  =
 \int_M(\frac{d\tilde{w}}{dt})^2 dM
 = -\int_M \tilde{w}\left( 
\frac{d^2\tilde{w}}{dt^2} + \frac{d\tilde{w}}{dt}\Delta_0 r\right)dM.$$
This is just the  same expression as in the proof  using the spherical 
geodesic coordinates. We choose to explicitly use the coordinate chart to 
see that, if $r_0$ is not smaller than $\mathrm{inj}(p_0)$, we still can 
get $(\ref{negative1})\leq 0$ as in \cite{FMS}, by using $d_{\xi}$ 
instead $r_0$. If the min-max formula is valid on  domains $\bar{M}$ with 
less regular  $\partial M$, we  can obtain the same conclusion
in Theorem \ref{Theorem 6.8} for  geodesic balls with radius exceeding
the injectivity radius.}
\end{remark}

%\smallskip
\begin{remark}   {\rm The study of the spectrum of the Laplacian with respect 
to a metric connection $\nabla$ is only interesting if 
$B(X,Y)=\nabla_XY-\nabla_X^0Y$ has a nonzero symmetric part, as it is the 
present case. For instance, connections with skew torsion (see  definition 
in \cite{ACF}) have the same Laplacian as the Levi-Civita connection one.   }  
\end{remark}



%\bibliographystyle{elsarticle-num} 


\section*{}

\end{document}